\documentclass[11pt]{amsart}
\usepackage[left=1.25in,right=1.25in]{geometry}
\usepackage{amsthm}
\usepackage{amsmath}
\usepackage{mathtools}
\usepackage{amssymb}
\usepackage{hyperref}
\usepackage{tikz}
\usetikzlibrary{decorations.markings}
\usepackage{enumitem}
\usepackage{amsfonts}
\usepackage{hyperref}
\usepackage{todonotes}
\usepackage[utf8]{inputenc}

\newcommand\precdot{\mathrel{\ooalign{$\prec$\cr
  \hidewidth\raise0.001ex\hbox{$\cdot\mkern0.6mu$}\cr}}}

\newtheorem{theorem}{Theorem}[section]
\newtheorem{prop}[theorem]{Proposition}
\newtheorem{conj}[theorem]{Conjecture}
\newtheorem{lemma}[theorem]{Lemma}
\newtheorem{cor}[theorem]{Corollary}

\theoremstyle{definition}

\newtheorem{defin}[theorem]{Definition}

\theoremstyle{remark}
\newtheorem*{remark}{Remark}

\DeclareMathOperator{\supp}{Supp}
\DeclareMathOperator{\aut}{Aut}

\newcommand{\mc}{\mathcal}

\newcommand{\Inv}{\mathrm{Inv}}

\title{On automorphisms of undirected Bruhat graphs}

\author{Christian Gaetz}
\thanks{C.G. is supported by a National Science Foundation Postdoctoral Research Fellowship under grant DMS-2103121.}
\address{Department of Mathematics, Harvard University, Cambridge, MA.}
\email{\href{mailto:crgaetz@gmail.com}{{\tt crgaetz@gmail.com}}}

\author{Yibo Gao}
\address{Department of Mathematics, Massachusetts Institute of Technology, Cambridge, MA.}
\email{\href{mailto:gaoyibo@mit.edu}{{\tt gaoyibo@mit.edu}}}

\date{\today}

\begin{document}
\begin{abstract}
The \emph{(directed) Bruhat graph} $\widehat{\Gamma}(u,v)$ has the elements of the Bruhat interval $[u,v]$ as vertices, with directed edges given by multiplication by a reflection. Famously, $\widehat{\Gamma}(e,v)$ is regular if and only if the Schubert variety $X_v$ is smooth, and this condition on $v$ is characterized by pattern avoidance. In this work, we classify when the \emph{undirected} Bruhat graph $\Gamma(e,v)$ is \emph{vertex-transitive}; surprisingly this class of permutations is also characterized by pattern avoidance and sits nicely between the classes of smooth permutations and self-dual permutations. This leads us to a general investigation of automorphisms of $\Gamma(u,v)$ in the course of which we show that \emph{special matchings}, which originally appeared in the theory Kazhdan--Lusztig polynomials, can be characterized as certain $\Gamma(u,v)$-automorphisms which are conjecturally sufficient to generate the orbit of $e$ under $\aut(\Gamma(e,v))$.
\end{abstract}
\maketitle

\section{Introduction}
\label{sec:intro}

The \emph{(directed) Bruhat graph} $\widehat{\Gamma}$ of a Coxeter group $W$ is the directed graph with vertex set $W$ and directed edges $w \to wt$ whenever $\ell(wt)>\ell(w)$; we write $\widehat{\Gamma}(u,v)$ for its restriction to a Bruhat interval $[u,v] \subset W$. These graphs appear ubiquitously in the combinatorics of Coxeter groups and Bruhat order \cite{dyer-bruhat-graph}, the topology of flag, Schubert, and Richardson varieties as the GKM-graph for the natural torus action \cite{GKM, guillemin-holm-zara}, and in the geometry of these varieties and related algebra, for example in the context of Kazhdan--Lusztig polynomials \cite{blundell2021towards, Brenti-combinatorial-formula, davies2021advancing, Dyer-hecke-algebras}.

In all of these contexts, the directions of the edges, and sometimes additional edge labels, are centrally important. In this work, however, we study the associated \emph{undirected} graphs $\Gamma(u,v)$. In particular, from the perspective of the undirected graph, it is very natural to study graph automorphisms (in contrast, the directed Bruhat graph $\widehat{\Gamma}$ has very few automorphisms \cite{Waterhouse}), and these automorphisms end up having close connections to previous work on smooth Schubert varieties \cite{lakshmibai-sandhya, Carrell-smoothness}, self-dual Bruhat intervals \cite{self-dual}, Billey--Postnikov decompositions \cite{billey-postnikov, richmond-slofstra-fiber-bundle}, and special matchings \cite{SM-advances}.

\subsection{Regular, vertex-transitive, and self-dual Bruhat graphs}

The following well-known theorem, combining results of Lakshmibai--Sandhya \cite{lakshmibai-sandhya} and Carrell--Peterson \cite{Carrell-smoothness}, helped establish the fundamentality of both the Bruhat graph and pattern avoidance conditions in the combinatorial and geometric study of Schubert varieties.

\begin{theorem}[Lakshmibai--Sandhya \cite{lakshmibai-sandhya}, Carrell--Peterson \cite{Carrell-smoothness}]
\label{thm:smoothness}
The following are equivalent for a permutation $w$ in the symmetric group $\mathfrak{S}_n$:
\begin{enumerate}
    \item[\normalfont(S1)] the Bruhat graph $\widehat{\Gamma}(w)$ is a regular graph,  
    \item[\normalfont(S2)] the permutation $w$ avoids the patterns $3412$ and $4231$, 
    \item[\normalfont(S3)] the poset $[e,w]$ is rank-symmetric, and
    \item[\normalfont(S4)] the Schubert variety $X_w$ is smooth.
\end{enumerate}
\end{theorem}

In light of (S3), it is natural to ask whether $[e,w]$ is in fact self-dual as a poset when $X_w$ is smooth. This turns out to not always be the case, but the smaller class of self-dual intervals also admits a nice characterization by pattern avoidance:

\begin{theorem}[G.--G. \cite{self-dual}]
\label{thm:self-dual}
The following are equivalent for a permutation $w \in \mathfrak{S}_n$:
\begin{enumerate}
    \item[\normalfont(SD1)] the Bruhat interval $[e,w]$ is self-dual as a poset, and
    \item[\normalfont(SD2)] the permutation $w$ avoids the patterns $3412$ and $4231$ as well as $34521, 54123, 45321,$ and $54312$.
\end{enumerate}
\end{theorem}

In our first main theorem here, we characterize by pattern avoidance those permutations $w$ such that $\Gamma(e,w)$ is \emph{vertex-transitive}; this characterization implies that this class of permutations sits nicely between the classes of self-dual permutations (Theorem~\ref{thm:self-dual}) and smooth permutations (Theorem~\ref{thm:smoothness}).

\begin{theorem}
\label{thm:vertex-transitive}
The following are equivalent for a permutation $w \in \mathfrak{S}_n$:
\begin{enumerate}
    \item[\normalfont(VT1)] the undirected Bruhat graph $\Gamma(e,w)$ is a vertex-transitive graph,
    \item[\normalfont(VT2)] the permutation $w$ avoids the patterns $3412$ and $4231$ as well as $34521$ and $54123$.
\end{enumerate}
\end{theorem}

Since vertex-transitive graphs are necessarily regular, it is clear that the permutations from Theorem~\ref{thm:vertex-transitive} are a subset of those from Theorem~\ref{thm:smoothness}, and this is borne out by comparing conditions (S2) and (VT2). It is not at all conceptually clear, however, why the self-dual permutations of Theorem~\ref{thm:self-dual} should in turn be a subset of those from Theorem~\ref{thm:vertex-transitive}, even though this fact is easily seen by comparing conditions (VT2) and (SD2). A conceptual bridge between these two classes of permutations is provided by Conjecture~\ref{conj:orbit-is-interval}.

\begin{conj}
\label{conj:orbit-is-interval}
Let $w \in \mathfrak{S}_n$ and let $\mc{O}=\{\varphi(e) \mid \varphi \in \aut(\Gamma(e,w))\}$ be the orbit of the identity under graph automorphisms of $\Gamma(e,w)$, then
\[
\mc{O}=[e,v]
\]
for some $v \leq w$.
\end{conj}

Indeed, if $[e,w]$ is self-dual, then $w \in \mathcal{O}$, and so if Conjecture~\ref{conj:orbit-is-interval} holds we must have $\mathcal{O}=[e,w]$. That is, $\Gamma(w)$ must be vertex-transitive.

In the course of the proof of Theorem~\ref{thm:vertex-transitive} (Section~\ref{sec:vertex-transitive}) and the refinement of Conjecture~\ref{conj:orbit-is-interval} in Section~\ref{sec:orbits}, we are led to consider certain automorphisms of $\Gamma(u,v)$ arising from perfect matchings on the Hasse diagram of $[u,v]$. That these automorphisms are the same thing as the previously well-studied \emph{special matchings} on $[u,v]$ is the subject of our second main theorem.

\subsection{Special matchings and Bruhat automorphisms}
\emph{Special matchings} (see the definition in Section~\ref{sec:intro-SM}) on Bruhat intervals were introduced \cite{SM-original, SM-advances} because they can be used to define a recurrence for \emph{Kazhdan--Lusztig $R$-polynomials} \cite{kazhdan-lusztig-polynomials} which allows for the resolution of the \emph{Combinatorial Invariance Conjecture} in the case of lower intervals $[e,w]$. These matchings are intended to generalize many of the combinatorial properties of the matching on $W$ induced by multiplication by a simple reflection $s$. Special matchings on Bruhat intervals and related posets have since found several other combinatorial and topological applications and been generalized in several ways \cite{SM-topology, SM-diamonds, SM-zircon}, and special matchings on lower Bruhat intervals have been completely classified \cite{SM-lower-classification}.

In Theorem~\ref{thm:SM-equals-automorphism-classical-type} and Conjecture~\ref{conj:SM-equals-automorphism-general} below we give a new characterization of special matching of Bruhat intervals $[u,v]$ in terms of automorphisms of $\Gamma(u,v)$. This characterization is notable because it expresses the special matching condition, originally formulated as a condition only on Bruhat covers, as a condition on the global structure of the undirected Bruhat graph.

A Coxeter group $W$ is called \emph{right-angled} if every pair of simple generators either commutes or generates an infinite dihedral group.

\begin{theorem}
\label{thm:SM-equals-automorphism-classical-type}
Let $W$ be a right-angled Coxeter group or the symmetric group and let $u \leq v$ be elements of $W$. Then a perfect matching of the Hasse diagram of $[u,v]$ is a special matching if and only if it is an automorphism of $\Gamma(u,v)$. 
\end{theorem}

\begin{conj}
\label{conj:SM-equals-automorphism-general}
Theorem~\ref{thm:SM-equals-automorphism-classical-type} holds for arbitrary Coxeter groups $W$.
\end{conj}

\subsection{Outline}

In Section~\ref{sec:background}, we cover background and definitions relating to Coxeter groups, Bruhat order and Bruhat graphs, Billey--Postnikov decompositions, and special matchings. In Section~\ref{sec:automorphisms} we describe several sources of automorphisms of $\Gamma(u,v)$: left and right multiplication by a simple generator $s$ via the Lifting Property, and \emph{middle multiplication}, which we introduce, when the interval admits a particularly nice Billey--Postnikov decomposition. In Section~\ref{sec:vertex-transitive} we prove Theorem~\ref{thm:vertex-transitive}, classifying vertex-transitive intervals $[e,w]$. In Section~\ref{sec:orbits} we give a more precise version of Conjecture~\ref{conj:orbit-is-interval} in terms of \emph{almost reducible decompositions} and some partial results towards resolving the conjecture. Section~\ref{sec:automorphisms} proves Theorem~\ref{thm:SM-equals-automorphism-classical-type} and one direction of Conjecture~\ref{conj:SM-equals-automorphism-general}, establishing a close connection between automorphisms of $\Gamma(u,v)$ and special matchings on $[u,v]$. The proof of Theorem~\ref{thm:SM-equals-automorphism-classical-type} relies on a structural property of Bruhat order, the existence of upper bounds of \emph{butterflies}, which may be of independent interest. This property is discussed and proven for the symmetric group and right-angled Coxeter groups in Section~\ref{sec:butterflies}. 

\section{Background and definitions}
\label{sec:background}
\subsection{Coxeter groups and reflections}
We refer the reader to \cite{bjorner-brenti} for basic definitions and background for Coxeter groups.

For a Coxeter group $W$ with simple generators $S=\{s_1,\ldots,s_r\}$ and an element $w \in W$, an expression $w=s_{i_1}\cdots s_{i_{\ell}}$ is a \emph{reduced word} of $w$ if it is of minimal length, and in this case $\ell=\ell(w)$ is the \emph{length} of $w$. The \emph{reflections} $T$ are the $W$-conjugates of the simple reflections. The \emph{(left) inversion set} of $w$ is
\[
T_L(w) \coloneqq \{ t \in T \mid \ell(tw)<\ell(w)\},
\]
and the \emph{(left) descent set} of $w$ is $D_L(w)\coloneqq S \cap T_L(w)$. \emph{Right} inversion and descent sets $T_R(w),D_R(w)$ are defined analogously, using instead right multiplication by $t$. It is not hard to see that $\ell(w)=|T_L(w)|=|T_R(w)|$.

Given $J \subseteq S$, the \emph{parabolic subgroup} $W_J$ is the subgroup of $W$ generated by $J$, viewed as a Coxeter group with simple generators $J$. Each coset $wW_J$ for $W_J$ in $W$ contains a unique element $w^J$ of minimal length, and this determines a decomposition $w=w^Jw_J$ with $w_J \in W_J$ and $\ell(w)=\ell(w^J)+\ell(w_J)$. The set $W^J \coloneqq \{w^J \mid w \in W\}$ is the \emph{parabolic quotient} of $W$ with respect to $J$, and has the following alternative description:
\[
W^J = \{ u \in W \mid D_R(u) \cap J = \emptyset \}.
\]

If $W$ is finite, it contains a unique element $w_0$ of maximum length, and the image $w_0^J$ of $w_0$ in any parabolic quotient is the unique longest element of $W^J$. We write $w_0(J)$ for the longest element of the parabolic subgroup $W_J$.

\subsection{Bruhat graphs and Bruhat order}

The \emph{directed Bruhat graph} $\widehat{\Gamma}$ of $W$ is the directed graph with vertex set $W$ and directed edges $w \to wt$ whenever $t$ is a reflection with $\ell(wt)>\ell(w)$. Note that, since $T$ is closed under conjugation, the ``left" and ``right" versions of $\widehat{\Gamma}$ in fact coincide. The \emph{(undirected) Bruhat graph} $\Gamma$ is the associated simple undirected graph.  The directed graph $\widehat{\Gamma}$ is much more commonly considered in the literature, and often called ``the Bruhat graph" but, since our focus in this work is on the undirected graph $\Gamma$, when directedness is not specified we mean the undirected graph.

The \emph{(strong) Bruhat order} $(W,\leq)$ is the partial order on $W$ obtained by taking the transitive closure of the relation determined by $\widehat{\Gamma}$. We write $[u,v]$ for the interval $\{w \in W \mid u \leq w \leq v\}$ in Bruhat order. For $u \leq v$, we write $\widehat{\Gamma}(u,v)$ and $\Gamma(u,v)$ for the restrictions of $\widehat{\Gamma}, \Gamma$ to the vertex set $[u,v]$; when $u$ is the identity element $e$, we sometimes write simply $\widehat{\Gamma}(v)$ and $\Gamma(v)$.

The following fundamental properties of Bruhat order will be of use throughout the paper.

\begin{prop}[Exchange Property]
\label{prop:exchange-property}
Let $w \in W$ and $t \in T$ be such that $\ell(wt)<\ell(w)$, and let $s_{i_1}\cdots s_{i_k}$ be any (not-necessarily-reduced) expression for $w$, then for some $j$ we have
\[
wt=s_{i_1}\cdots \widehat{s_{i_j}} \cdots s_{i_k}.
\]
\end{prop}

\begin{prop}[Subword Property]
\label{prop:subword-property}
Let $u,v \in W$, then $u \leq v$ if and only if some (equivalently, every) reduced word for $v$ contains a reduced word for $u$ as a subword.
\end{prop}

\begin{prop}[Lifting Property]
\label{prop:lifting-property}
Let $u \leq v$. If $s \in D_L(v) \setminus D_L(u)$, then $su<v$ and $u<sv$; analogously, if $s \in D_R(v) \setminus D_R(u)$, then $us<v$ and $u<vs$.
\end{prop}

\begin{prop}
\label{prop:monotonicity-of-projection}
Let $u \leq v$, then for any $J \subseteq S$ we have $u^J \leq v^J$.
\end{prop}

\subsection{Billey--Postnikov decompositions}

For $w \in W$, we write $\supp(w)$ for the \emph{support} of $w$: the set of simple reflections appearing in some (equivalently, every) reduced word for $w$.

\begin{defin}[Billey--Postnikov \cite{billey-postnikov}, Richmond--Slofstra \cite{richmond-slofstra-fiber-bundle}]
\label{def:BP-decomposition}
Let $W$ be a Coxeter group and $J \subseteq S$, the parabolic decomposition $w=w^Jw_J$ of $w$ is a \emph{Billey--Postnikov decomposition} or \emph{BP-decomposition} if
\[
\supp(w^J) \cap J \subseteq D_L(w_J).
\]
\end{defin}

BP-decompositions were introduced by Billey and Postnikov in \cite{billey-postnikov} in the course of their study of pattern avoidance criteria for smoothness of Schubert varieties in all finite types. The following characterizations of BP-decompositions will be useful to us.

\begin{prop}[Richmond--Slofstra \cite{richmond-slofstra-fiber-bundle}]
\label{prop:BP-equivalences}
For $w \in W$ and $J \subseteq S$, the following are equivalent:
\begin{enumerate}
    \item $w=w^Jw_J$ is a BP-decomposition,
    \item the multiplication map $\left( [e,w^J] \cap W^J \right) \times [e,w_J] \to [e,w]$ is a bijection,
    \item $w_J$ is the maximal element of $W_J \cap [e,w]$.
\end{enumerate}
\end{prop}

\subsection{Special matchings}
\label{sec:intro-SM}

The \emph{Hasse diagram}, denoted $H(P)$, of a poset $P$ is the undirected graph with vertex set $P$ and edges $(x,y)$ whenever $x \lessdot_P y$ is a cover relation in $P$. Note that the Hasse diagram $H(W)$ of Bruhat order on $W$ is a (non-induced) subgraph of $\Gamma$, as the two graphs share the vertex set $W$, but $H(W)$ contains only those edges $(x,y)$ of $\Gamma$ such that $| \ell(x) - \ell(y) | =1$. A \emph{perfect matching} of a graph $G$ is a fixed-point-free involution $M:G \to G$ such that $(x,M(x))$ is an edge of $G$ for all $x \in G$.

\begin{defin}[Brenti \cite{SM-original}, Brenti--Caselli--Marietti \cite{SM-advances}]
\label{def:special-matching}
A perfect matching $M$ on the Hasse diagram of a poset $P$ is a \emph{special matching} if, for every cover relation $x \lessdot_P y$, either $M(x)=y$ or $M(x) <_P M(y)$.
\end{defin}

The following general property of special matchings will be useful.

\begin{prop}[Brenti--Caselli--Marietti \cite{SM-advances}]
\label{prop:sm-restricts-to-interval}
Let $x<_P y$ and let $M$ be a special matching on a poset $P$. If $x<_P M(x)$ and $M(y)<_P y$, then $M$ restricts to a special matching on the poset $[x,y]$.
\end{prop}

\section{Automorphisms and middle multiplication}
\label{sec:automorphisms}
\subsection{Elementary automorphisms}\label{sub:elementary-automorphisms}
The following result of Waterhouse shows that $\widehat{\Gamma}$ has no nontrivial automorphisms as a directed graph.

\begin{theorem}[Waterhouse \cite{Waterhouse}]
Let $W$ be an irreducible Coxeter group which is not dihedral, then $\aut((W,\leq))$ (equivalently, $\aut(\widehat{\Gamma})$) is generated by the graph automorphisms of the Dynkin diagram of $W$ and the group inversion map on $W$.
\end{theorem}

In this paper we study the much richer sets of automorphisms of $\Gamma$ and particularly of its subgraphs $\Gamma(u,v)$. The simplest automorphisms of $\Gamma$ not coming from automorphisms of $\widehat{\Gamma}$ follow from the Lifting Property:

\begin{prop}
\label{prop:right-left-multiplication-give-autos}
Let $u,v \in W$ with $u \leq v$ and suppose $s \in D_L(v) \setminus D_L(u)$, then left multiplication $L_s: W \to W$ by $s$ restricts to a graph automorphism of $\Gamma(u,v)$. Similarly, if $s \in D_R(v) \setminus D_R(u)$, then right multiplication $R_s: W \to W$ by $s$ restricts to a graph automorphism of $\Gamma(u,v)$.
\end{prop}
\begin{proof}
As left and right multiplication commute, and as $\Gamma$ may equivalently be defined either in terms of left or right multiplication by reflections, it is clear that $L_s,R_s$ define automorphisms of $\Gamma$. We just need to check that, under the hypotheses, they preserve the Bruhat interval $[u,v]$.

Suppose $s \in D_L(v) \setminus D_L(u)$ and $x \in [u,v]$. If $s \in D_L(x)$, then by the Lifting Property (Proposition~\ref{prop:lifting-property}) we have $u \leq sx$, thus, since $sx<x\leq v$ we have $sx \in [u,v]$. Similarly, if $s \not \in D_L(x)$, then the Lifting Property implies that $u\leq x<sx \leq v$, so again $sx \in [u,v]$. The case of right multiplication is exactly analogous.
\end{proof}

We say the element $w\in W$ has a \emph{disjoint support decomposition} if it may be expressed as a nontrivial product $w=w'w''$ with $\supp(w') \cap \supp(w'') = \emptyset$ (note that, in this case, we have $w'=w^J$ and $w''=w_J$ with $J=\supp(w'')$). 

\begin{prop}
\label{prop:disjoint-support}
Let $w=w'w''$ be a disjoint support decomposition, then:
\begin{align*}
    \widehat{\Gamma}(w) &\cong \widehat{\Gamma}(w') \times \widehat{\Gamma}(w''), \\
    \Gamma(w) &\cong \Gamma(w') \times \Gamma(w''),\\
    [e,w] &\cong [e,w'] \times [e,w'']. 
\end{align*}
In each case, the isomorphism is given by group multiplication.
\end{prop}
\begin{proof}
The latter two assertions follow from the first. The first assertion can be easily seen by choosing a reduced word $w=s_{i_1}\cdots s_{i_k} s_{i_{k+1}} \cdots s_{i_{\ell}}$ such that $s_{i_1}\cdots s_{i_k}=w'$ and $s_{i_{k+1}} \cdots s_{i_{\ell}}=w''$ and applying the Subword Property and Exchange Property.
\end{proof}

Proposition~\ref{prop:disjoint-support} implies in particular that when $w=w'w''$ is a disjoint support decomposition, $\Gamma(w),\widehat{\Gamma}(w),$ and $[e,w]$ have automorphisms induced by automorphisms for $w',w''$.

\subsection{Middle multiplication}

A key observation of this work is that, when $w$ admits a particularly nice parabolic decomposition which is sufficiently close to a disjoint support decomposition, $\Gamma(w)$ has additional automorphisms coming from \emph{middle multiplication}. These automorphisms will be key to the classification in Section~\ref{sec:vertex-transitive} of vertex-transitive Bruhat graphs $\Gamma(w)$ in the symmetric group.

\begin{prop}
\label{prop:middle-mult-is-automorphism}
Suppose $w=w^Jw_J$ is a BP-decomposition of $W$ and in addition we have $\supp(w^J) \cap \supp(w_J) =\{s\}$, then the middle multiplication map 
\[
\phi: x \mapsto x^J s x_J, 
\]
is an automorphism of the Bruhat graph $\Gamma(w)$.
\end{prop}
\begin{proof}
Suppose without loss of generality that $J=\supp(w_J)$.

By Proposition~\ref{prop:BP-equivalences}, we have that $y \mapsto (y^J,y_J)$ is a bijection between $[e,w]$ and $([e,w^J]\cap W^J) \times [e,w_J]$. We need to verify that if $x=yt$ for $x,y \in [e,w]$ and $t \in T$ then $\phi(x)=\phi(y)t'$ for some $t' \in T$.

If $x_J=y_J$, then
\begin{align*}
    \phi(x)&=x^Jsy_J=\phi(x)y_J^{-1}sy_J \\
    \phi(y)&=\phi(y)y_J^{-1}sy_J.
\end{align*}
Thus, since $\phi$ acts on $x,y$ just by right multiplication, we are done. The case $x^J=y^J$ is similar.

Thus we are left with the case where $x^J \neq y^J$ and $x_J \neq y_J$. Suppose without loss of generality that $\ell(y)>\ell(x)$. Fix some reduced word $y=s_{i_1} \cdots s_{i_{\ell}}$ such that $s_{i_1} \cdots s_{i_k}=y^J$ and $s_{i_{k+1}} \cdots s_{i_{\ell}}=y_J$. By the Exchange Property (Proposition~\ref{prop:exchange-property}), we have a not-necessarily-reduced word $x=s_{i_1} \cdots \widehat{s_{i_j}} \cdots s_{i_{\ell}}$ for some $1 \leq j \leq \ell$. If $j > k$, so that we are deleting a letter from $y_J$, then $z=s_{i_k} \cdots \widehat{s_{i_j}} \cdots s_{i_{\ell}}$ still lies in $W_J$, and so uniqueness of parabolic decompositions implies that $z=x_J$ and $x^J=y^J$, a contradiction. Thus it must be that $j \leq k$. Furthermore, we must have $z'=s_{i_1} \cdots \widehat{s_{i_j}} \cdots s_{i_{k}} \not \in W^J$. By definition, this means that $z'$ has some right descent from $J$, and since $\supp(z') \subseteq \supp(y^J) \subseteq \supp(w^J)$ and $\supp(w^J) \cap J = \{s\}$, it must be that $s \in D_R(z')$. This implies that $(x^J,x_J)=(z's,sy_J)$. Now,
\begin{align*}
    \phi(x)&=x^Jsx_J=z'sy_J \\
    \phi(y)&=y^Jsy_J.
\end{align*}
So $\phi(x)\phi(y)^{-1}=z'(y^J)^{-1}$. We know $z'=y^Jt''$ for some $t'' \in T$ by the Exchange Property, so this expression becomes 
\[
y^Jt''(y^J)^{-1} \in T.
\]
\end{proof}

\section{Vertex transitive Bruhat graphs}
\label{sec:vertex-transitive}
To prove Theorem~\ref{thm:vertex-transitive}, we introduce another condition:
\begin{itemize}
    \item[(VT3)] \textit{the element $w$ is almost-polished},
\end{itemize}
and show its equivalence to (VT1) and (VT2).
\begin{defin}\label{def:almost-polished}
Let $(W,S)$ be a finite Coxeter system. An element $w\in W$ is \emph{almost-polished} if there exists pairwise disjoint subsets $S_1,\ldots,S_k\subset S$ such that each $S_i$ is a connected subset of the Dynkin diagram and coverings $S_i=J_i\cup J_i'$ with $i=1,\ldots,k$ so that
\[w=\prod_{i=1}^k w_0(J_i)w_0(J_i\cap J_i')w_0(J_i').\]
\end{defin}
Note that if we reorder the $S_i$'s, a possibly different almost-polished element can be obtained. Unlike \emph{polished elements} \cite{self-dual}, for almost-polished elements, we do not require that $J_i\cap J_i'$ is totally disconnected.  

We then prove Theorem~\ref{thm:vertex-transitive} via (VT1)$\Rightarrow$(VT2)$\Rightarrow$(VT3)$\Rightarrow$(VT1).

\subsection{(VT1)$\Rightarrow$(VT2)}
For a simple graph $G$ and $v\in V(G)$, let $N_d(v)$ be the set of vertices with distance exactly $d$ from $v$. In particular, $N_0(v)=\{v\}$ and for the Bruhat graph $\Gamma(w)$, $N_1(w)=\{wt_{ij}\:|\: w(i)>w(j)\}$ where $t_{ij}$ is the transposition $(i\ j)$ with $i<j$. For $\varphi\in\aut(G)$, it is clear that if $\varphi(u)=v$, then $\varphi(N_d(u))=N_d(v)$ for all $d=1,2,\ldots$. 

We start with a simple lemma concerning $\aut(\Gamma(w))$, which intuitively says that $\varphi\in\aut(\Gamma(w))$ ``preserves triangles" in $N_1(w)$ and $N_1(e)$. 
\begin{lemma}\label{lem:preserve-triangle}
Let $\varphi\in\aut(\Gamma(w))$ such that $\varphi(w)=e$.
\begin{enumerate}
\item If $i<j<k$ and $w(i)>w(j)>w(k)$, then for some $a<b<c$ with $w\geq t_{ac}$, \[\varphi(\{wt_{ij},wt_{ik},wt_{jk}\})=\{t_{ab},t_{ac},t_{bc}\}.\]
\item If $a<b<c$ and $w\geq t_{ac}$, then for some $i<j<k$ and $w(i)>w(j)>w(k)$,
\[\varphi^{-1}\big(\{t_{ab},t_{ac},t_{bc}\}\big)=\{wt_{ij},wt_{ik},wt_{jk}\}.\]
\end{enumerate}
\end{lemma}
\begin{proof}
Recall that for $a<b<c$, $t_{ac}\geq t_{ab}$ and $t_{ac}\geq t_{bc}$.

For (1), we notice that the three elements $wt_{ij},wt_{ik},wt_{jk}\in N_1(w)$ have two common neighbors in $wt_{ij}t_{jk}=wt_{ik}t_{ij}, wt_{jk}t_{ij}=wt_{ik}t_{jk}\in N_2(w)$. In order for $\varphi(wt_{ij})=t_{x_1y_1}$, $\varphi(wt_{ik})=t_{x_2y_2}$ and $\varphi(wt_{jk})=t_{x_3y_3}$ to have two common neighbors in $N_2(e)$, $\{x_1,y_1\}$, $\{x_2,y_2\}$, $\{x_3,y_3\}$ must pairwise intersect, so they must be $\{a,b\}$, $\{a,c\}$, $\{b,c\}$ in some order, for some $a<b<c$. For (2), the exact same reasoning works by analysing $N_2(w)$. 
\end{proof}

\begin{proof}[Proof of implication (VT1)$\Rightarrow$(VT2)]
Assume $\Gamma(w)$ is vertex-transitive. In particular, it is a regular graph so Theorem~\ref{thm:smoothness} implies that $w$ avoids $3412$ and $4231$. The other two patterns $34521$ and $54123$ are inverse to each other. Since $\Gamma(w)$ and $\Gamma(w^{-1})$ are isomorphic, it suffices to show that $w$ avoids $34521$. Now for the sake of contradiction, let $w$ contain the pattern $34521$ at indices $a_1<\cdots<a_5$, and let $\varphi\in\aut(\Gamma(w))$ such that $\varphi(w)=e$.

Let $\varphi(wt_{a_4a_5})=t_{xy}$ where $x<y$. Note that $w$ contains the pattern $321$ at indices $a_i<a_4<a_5$ for $i=1,2,3$. By Lemma~\ref{lem:preserve-triangle}, let $\varphi(\{wt_{a_ia_4},wt_{a_ia_5},wt_{a_4a_5}\})$ be the three transpositions with indices in $x,y$ and $c_i$. We now view the indices $a_1,a_2,a_3$ and $c_1,c_2,c_3$ with symmetric roles and divide into the following cases.

Case 1: $x<c_1<c_2<y$. Since $w\geq t_{xy}>t_{xc_2}$, by Lemma~\ref{lem:preserve-triangle}(2), $\varphi^{-1}(\{t_{xc_1},t_{xc_2},t_{c_1c_2}\})=\{wt_{ij},wt_{ik},wt_{jk}\}$ for some $i<j<k$. We know that $\varphi^{-1}(t_{xc_1})=wt_{a_1a_4}$ or $wt_{a_1a_5}$ and $\varphi^{-1}(t_{xc_2})=wt_{a_2a_4}$ or $w_{t_{a_2a_5}}$. In order for them to have common indices, we must have either $(i,j,k)=(a_1,a_2,a_4)$ or $(i,j,k)=(a_1,a_2,a_5)$. But $wt_{a_1a_2}>w$ is not a vertex in $\Gamma$, resulting in a contradiction.

Case 2: $x<c_1<y<c_2$. The positioning of $c_2$ implies $w\geq t_{xc_2}\geq t_{c_1c_2}$. By Lemma~\ref{lem:preserve-triangle}(2), $\varphi^{-1}(\{t_{c_1y},t_{c_1c_2},t_{yc_2}\})=\{wt_{ij},wt_{ik},wt_{jk}\}$ for some $i<j<k$. We know that $\varphi^{-1}(t_{c_1y})=wt_{a_1a_4}$ or $wt_{a_1a_5}$ and $\varphi^{-1}(t_{yc_2})=wt_{a_2a_4}$ or $w_{t_{a_2a_5}}$. This leads to the same contradiction as above due to $wt_{a_1a_2}>w$.

By symmetry, Case 1 and Case 2 together cover all the situations where one of $\{c_1,c_2,c_3\}$ lie in the open interval $(x,y)$. We can then assume that $c_i<x$ or $c_i>y$ for $i=1,2,3$. By symmetry and the pigeonhole principle, we can assume that at least two of $\{c_1,c_2,c_3\}$  are greater than $y$.

Case 3: $x<y<c_1<c_2$. Again, we have $w>t_{yc_2}$. By Lemma~\ref{lem:preserve-triangle}(2), we consider the indices $y<c_1<c_2$ and apply the exact same arguments as in Case 1 and Case 2 to $\varphi^{-1}(\{t_{c_1y},t_{yc_2},t_{c_1c_2}\})=\{wt_{ij},wt_{ik},wt_{jk}\}$ to derive a contradiction.
\end{proof}

\subsection{(VT2)$\Rightarrow$(VT3)}
Throughout this section, assume $w\in \mathfrak{S}_n$ avoids the four permutations, $3412$, $4231$, $34521$ and 54123, in (VT2). We view permutations as their permutation matrix, with indices increasing from left to right and from top to bottom. We apply the standard decomposition techniques of smooth permutations, as in \cite{Gasharov,Oh-Postnikov-Yoo} and especially Section 3.2 of \cite{self-dual}, of which we follow the notations. Our goal is to write $w$ as a product of $w_0(K)$'s, for subsets $K\subset S=\{s_1,\ldots,s_{n-1}\}$, via induction.

Consider the top-left corner region
\[C=\{(a,w(a))\:|\: 1\leq a\leq w^{-1}(1),\ 1\leq w(a)\leq w(1)\}\]
which is rectangle formed by $(1,w(1))$ and $(w^{-1}(1),1)$. Let $C=\{(c_1,w(c_1)),\ldots,(c_t,w(c_t))\}$ where $c_1<\cdots<c_t$. Let $K_1=\{s_1,\ldots,s_{t-1}\}$. Since $w$ avoids $4231$, $w(c_1)>\cdots>w(c_t)$. Also consider the rectangle to the right of $C$ and on the bottom of $C$:
\begin{align*}
R=&\{(a,w(a))\:|\: 1<a<w^{-1}(1),\ w(a)>w(1)\}, \\
L=&\{(a,w(a))\:|\: 1<w(a)<w(1),\ a>w^{-1}(1)\}.
\end{align*}
The positioning of these regions can be seen in Figure~\ref{fig:smooth-permutation}.
\begin{figure}[h!]
\centering
\begin{tikzpicture}[scale=0.4]
\draw(0,0)--(7,0);
\draw(0,0)--(0,-7);
\draw(0,-4)--(7,-4);
\draw(4,0)--(4,-7);
\node at (4,0) {$\bullet$};
\node at (0,-4) {$\bullet$};
\node at (2,-5.5) {$L$};
\node at (5.5,-2) {$R$};
\node at (1,-1) {$C$};
\node at (2.4,-2.4) {$\bullet$};
\node at (3.2,-1.2) {$\bullet$};
\node at (1.2,-3.2) {$\bullet$};
\node[left] at (0,0) {$c_1$};
\node[left] at (0,-1.2) {$c_2$};
\node[left] at (0,-2.4) {$\vdots$};
\node[left] at (0,-4) {$c_t$};
\end{tikzpicture}
\caption{Analyzing the structure of smooth permutations.}
\label{fig:smooth-permutation}
\end{figure}
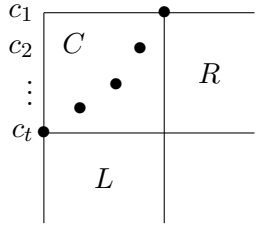
Since $w$ avoids $3412$, at least one of $R$ and $L$ is empty. If both are empty, we say that $w$ is of type n with parameter $K_2$. If $R$ is nonempty, we say that $w$ is of type r; and if $L$ is nonempty, we say that $w$ is of type l. Notice that if $w$ is of type n, $w':=w\cdot w_0(K_1)=w_0(K_1)\cdot w$ also avoids the four patterns in (VT2), with $\supp(w')\subset S\setminus K_1$. We call $w'$ the \emph{one-step reduction} of $w$ and we can then straightforwardly deduce that $w$ is almost-polished by $w'$ being almost-polished. We will come back to this later.

The case of type l and type r are dual to each other by taking inverses. So for now, we assume that $w$ is of type r, i.e. $L=\emptyset$, and further study its structure. Divide $R=R_1\sqcup\cdots\sqcup R_{t-1}$ where $R_i=\{(a,w(a))\:|\: c_i<a<c_{i+1}\}$. If $R_1=\cdots=R_{t-2}=\emptyset$, then $R_{t-1}\neq\emptyset$ and we say that $w$ is of type r0 with parameter $K_1$ and let $w'=w_0(K_1)\cdot w$ be the \emph{one-step reduction} of $w$. 

Here we use the condition that $w$ avoids $34521$ and $54123$. Since $w$ avoids $34521$, the coordinates in $R_1\cup\cdots\cup R_{t-2}$ must be decreasing, i.e. from top right to bottom left. At the same time, since $w$ avoids $4231$, we have that at most one of $R_1,\ldots,R_{t-2}$ is nonempty. If $R_p\neq\emptyset$, let $I_1=\{p+1,p+2,\ldots,t-1\}\subset K_1$ and say that $w$ is of type r1 with parameter $(K_1,I_1)$. Then let $w'=w_0(I_1)w_0(K_1)w$ be the \emph{one-step reduction} of $w$. Also note that within $R_{t-1}$ (which can be empty or nonempty), there are no coordinates to the left of any coordinates in $R_p$ since $w$ avoids $4231$. A visualization is shown in Figure~\ref{fig:one-step-reduction}.
\begin{figure}[h!]
\centering
\begin{tikzpicture}[scale=0.4]
\node at (4,0) {$\bullet$};
\node at (3.2,-1) {$\bullet$};
\node at (2.4,-2) {$\bullet$};
\node at (1.6,-4) {$\bullet$};
\node at (0.8,-5) {$\bullet$};
\node at (0,-7) {$\bullet$};
\draw(0,0)--(0,-10);
\draw(0,0)--(12,0);
\draw(4,0)--(4,-10);
\draw(0,-7)--(12,-7);
\draw[dashed](3.2,-1)--(12,-1);
\draw[dashed](2.4,-2)--(12,-2);
\draw[dashed](1.6,-4)--(12,-4);
\draw[dashed](0.8,-5)--(12,-5);
\node at (10.5,-0.5) {$R_1{=}\emptyset$};
\node at (10.5,-1.5) {$R_2{=}\emptyset$};
\node at (10.5,-3) {$R_3{\neq}\emptyset$};
\node at (10.5,-4.5) {$R_4{=}\emptyset$};
\node at (10.5,-6) {$R_{t-1}$};
\node at (5,-3.6) {$\bullet$};
\node at (6,-3.2) {$\bullet$};
\node at (7,-2.8) {$\bullet$};
\node at (8,-2.4) {$\bullet$};
\node at (2,-8.5) {$L=\emptyset$};
\draw[dashed](8,-2.4)--(8,-7);
\node at (6,-6) {$\emptyset$};
\end{tikzpicture}
\qquad\qquad\qquad
\begin{tikzpicture}[scale=0.4]
\node at (0,0) {$\bullet$};
\node at (0.8,-1) {$\bullet$};
\node at (1.6,-2) {$\bullet$};
\node at (4.0,-4) {$\bullet$};
\node at (3.2,-5) {$\bullet$};
\node at (2.4,-7) {$\bullet$};
\draw(0,0)--(0,-10);
\draw(0,0)--(12,0);
\draw(4,0)--(4,-10);
\draw(0,-7)--(12,-7);
\draw[dashed](1.6,-2)--(12,-2);
\draw[dashed](1.6,-2)--(1.6,-10);
\node at (5,-3.6) {$\bullet$};
\node at (6,-3.2) {$\bullet$};
\node at (7,-2.8) {$\bullet$};
\node at (8,-2.4) {$\bullet$};
\draw[dashed](3.2,-5)--(12,-5);
\draw[dashed](8,-2.4)--(8,-10);
\node at (2.8,-8.5) {$\emptyset$};
\node at (10,-3.5) {$\emptyset$};
\end{tikzpicture}
\caption{The permutation $w$ on the left and its one-step reduction $w'$ on the right, with type r1 and parameters $K_1=\{s_1,\ldots,s_5.\}$, $I_1=\{s_4,s_5\}$.}
\label{fig:one-step-reduction}
\end{figure}
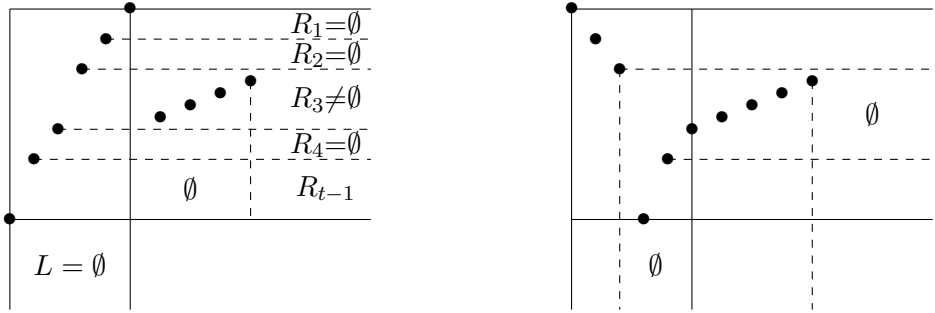
As a summary, the types of $w$ can be n, l0, l1, r0 or r1. 

\begin{lemma}\label{lem:one-step-reduction-avoid-patterns}
Let $w$ avoid $3412$, $4231$, $34521$ and $54123$ as above. Then the one-step reduction $w'$ of $w$ also avoids these four patterns.
\end{lemma}
\begin{proof}
We first deal with the critical case where $w$ is of type r1 (or l1) with parameters $K_1=\{s_1,\ldots,s_{t-1}\}$, $I_1=\{p+1,\ldots,t-1\}$ as above. Since $w'(1)=1,\ldots,w'(p)=p$, these indices cannot be involved in any of the four patterns of interest. At the same time, $w'$ restricted to the last $n-p$ indices equals $w$ restricted to the same indices, which avoids these four patterns. As a result, $w'$ avoid these patterns as well. The cases n, r0 and l0 follow from the same arguments.
\end{proof}

The one-step reduction $w'$ lives in a strictly smaller parabolic subgroup of $\mathfrak{S}_n$. And Lemma~\ref{lem:one-step-reduction-avoid-patterns} allows us to continue the reduction. In particular, if $w$ is of type n, l0 or r0, then $\supp(w')\subset S\setminus K_1$, and the next step of reduction can then be analyzed from scratch.

\begin{lemma}\label{lem:one-step-after-r1}
Let $w$ be of type r1 with one-step reduction $w'$ and parameters $K_1$ and $I_1$. Then $w'$ can be of type n, l0, r0 with parameter $K_2$, or l1 with parameters $K_2$ and $I_2$, where $I_1=K_1\cap K_2$ and there are no edges between $I_1$ and $I_2$ in the Dynkin diagram of $\mathfrak{S}_n$.
\end{lemma}
\begin{proof}
Keep the notation from above and let $K_1=\{s_1,\ldots,s_{t-1}\}$ and $I_1=\{s_{p+1},\ldots,s_{t-1}\}$. Let $|R_p|=q>0$, then by construction, $K_2=\{s_{p+1},\ldots,s_{t-1+q}\}$ so we immediately have $K_1\cap K_2=I_1$. For the new permutation $w'$, if it is of type r, then it is of type r0 since the new regions $R_1',\ldots,R_{t+q-p-2}'$, which are subsets of $R_1,\ldots,R_{p-1},R_{p+1},\ldots,R_{t-2}$, must be empty. See Figure~\ref{fig:one-step-reduction}.

The permutation $w'$ can be of type n or l. If it is of type l, by dividing $L'$ into $L_1'\sqcup\cdots\sqcup L_{t+q-p-1}'$ analogously as before, we see that $L_1'=\cdots=L_{t-p-1}'=\emptyset$ because these regions belong to $L=\emptyset$. By construction, this means $t\notin I_2$ so the consecutive intervals $I_1$ and $I_2$ do not have edges between them. 
\end{proof}

We are now ready to fully decompose $w$ and show that it is almost-polished.
\begin{proof}[Proof of implication (VT2)$\Rightarrow$(VT3)]
Let $w\in \mathfrak{S}_n$ avoid the four patterns of interest and keep the notations in this section. We use induction where the base cases $n=1,2$ are vacuously true. Let $w^{(1)}=w$ and continue to do one-step reduction of $w^{(i)}$ to obtain $w^{(i+1)}$, until $w^{(m)}$ is of type n, r0 or l0 whose one-step reduction equals $w^{(m+1)}$. Note that this is possible because the one-step reduction of type r1 or l1 never equals the identity. By Lemma~\ref{lem:one-step-after-r1}, as $i$ increases from $1$ to $m-1$, $w^{(i)}$ alternates between type r1 and type l1. Let $w^{(i)}$ have parameters $K_i$ and $I_i$ as above and let $S_1=K_1\cup K_2\cup\cdots K_{m-1}\cup K_m$. Here, the $K_i$'s and $I_i's$ are consecutive intervals ordered from left to right, respectively. Moreover by Lemma~\ref{lem:one-step-after-r1}, $K_i\cap K_{i+1}=I_i$ and the smallest index in $I_{i+1}$ is at least $2$ bigger than the largest index in $I_i$. See Figure~\ref{fig:KiIis} for an example of these intervals.
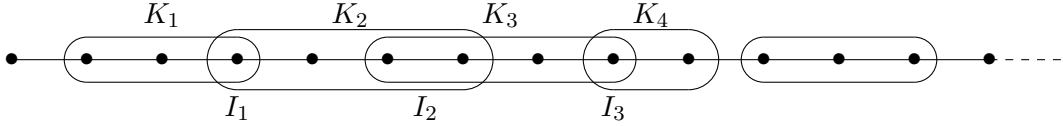
\begin{figure}[h!]
\centering
\begin{tikzpicture}[scale=1.0]
\draw(0,0)--(13,0);
\draw[dashed](13,0)--(14,0);
\node at (0,0) {$\bullet$};
\node at (1,0) {$\bullet$};
\node at (2,0) {$\bullet$};
\node at (3,0) {$\bullet$};
\node at (4,0) {$\bullet$};
\node at (5,0) {$\bullet$};
\node at (6,0) {$\bullet$};
\node at (7,0) {$\bullet$};
\node at (8,0) {$\bullet$};
\node at (9,0) {$\bullet$};
\node at (10,0) {$\bullet$};
\node at (11,0) {$\bullet$};
\node at (12,0) {$\bullet$};
\node at (13,0) {$\bullet$};

\draw(1,0.3)--(3,0.3);
\draw(1,-0.3)--(3,-0.3);
\draw (1,0.3) arc (90:270:0.3);
\draw (3,-0.3) arc (-90:90:0.3);

\draw(3,0.4)--(6,0.4);
\draw(3,-0.4)--(6,-0.4);
\draw (3,0.4) arc (90:270:0.4);
\draw (6,-0.4) arc (-90:90:0.4);

\draw(5,0.3)--(8,0.3);
\draw(5,-0.3)--(8,-0.3);
\draw (5,0.3) arc (90:270:0.3);
\draw (8,-0.3) arc (-90:90:0.3);

\draw(8,0.4)--(9,0.4);
\draw(8,-0.4)--(9,-0.4);
\draw (8,0.4) arc (90:270:0.4);
\draw (9,-0.4) arc (-90:90:0.4);

\draw(10,0.3)--(12,0.3);
\draw(10,-0.3)--(12,-0.3);
\draw (10,0.3) arc (90:270:0.3);
\draw (12,-0.3) arc (-90:90:0.3);

\node[above] at (2,0.3) {$K_1$};
\node[above] at (4.5,0.3) {$K_2$};
\node[above] at (6.5,0.3) {$K_3$};
\node[above] at (8.5,0.3) {$K_4$};
\node[below] at (3,-0.35) {$I_1$};
\node[below] at (5.5,-0.35) {$I_2$};
\node[below] at (8,-0.35) {$I_3$};
\end{tikzpicture}
\caption{The intervals in the reduction process.}
\label{fig:KiIis}
\end{figure}

Since $w^{(m)}$ is of type n or r1 or l1, $w^{(m+1)}=w^{(m)}w_0(K_m)$ or $w^{(m)}w_0(K_m)$ and also $\supp(w^{(m+1)})\subset S\setminus S_1$. Without loss of generality, assume $w=w^{(1)}$ is of type r1, then we have
\[w^{(m+1)}=\big(\cdots w_0(I_3)w_0(K_3)w_0(I_1)w_0(K_1)w w_0(K_2)w_0(I_2)w_0(K_4)w_0(I_4)\cdots\big)w_0(K_m) \]
or
\[w^{(m+1)}=w_0(K_m)\big(\cdots w_0(I_3)w_0(K_3)w_0(I_1)w_0(K_1)w w_0(K_2)w_0(I_2)w_0(K_4)w_0(I_4)\cdots\big).\]

Unpacking, we have the following possibilities:
\[w=\begin{cases}
\cdots w_0(K_{m-2})w_0(I_{m-2})(w^{(m+1)}w_0(K_m))w_0(I_{m-1})w_0(K_{m-1})\cdots\\
\cdots w_0(K_{m-2})w_0(I_{m-2})(w_0(K_m)w^{(m+1)})w_0(I_{m-1})w_0(K_{m-1})\cdots\\
\cdots w_0(K_{m-1})w_0(I_{m-1})(w^{(m+1)}w_0(K_m))w_0(I_{m-2})w_0(K_{m-2})\cdots\\
\cdots w_0(K_{m-1})w_0(I_{m-1})(w_0(K_m)w^{(m+1)})w_0(I_{m-2})w_0(K_{m-2})\cdots\\
\end{cases}.\]
As $w^{(m+1)}$ commutes with $w_0(K_i)$ and $w_0(I_i)$ where $i\leq m-1$, in all the possibilities above, we can move $w^{(m+1)}$ all the way to the left or all the way to the right. Moreover, $w_0(I_i)$ commutes with $w_0(K_j)$ if $j\neq i,i+1$. Thus, in all the possibilities above, we can move $w_0(I_i)$'s towards the middle, forming $w_0(I_1\cup I_2\cup\cdots \cup I_{m-1})$. 

Let $J_1=K_1\cup K_3\cup\cdots$ and $J_1'=K_2\cup K_4\cup\cdots$. We have $J_1\cap J_2=I_1\cup I_2\cup\cdots\cup I_{m-1}$, $w_0(J_1)=w_0(K_1)w_0(K_3)\cdots$, $w_0(J_1')=w_0(K_2)w_0(K_4)\cdots$. The above four possibilities can then be written as
\[w=\begin{cases}
w_0(J_1)w_0(J_1\cap J_1')w_0(J_1')w^{(m+1)}\\
w^{(m+1)}w_0(J_1)w_0(J_1\cap J_1')w_0(J_1')\\
w_0(J_1')w_0(J_1\cap J_1')w_0(J_1)w^{(m+1)}\\
w^{(m+1)}w_0(J_1')w_0(J_1\cap J_1')w_0(J_1)\\
\end{cases}.\]
As $\supp(w^{(m+1)})\subset S\setminus S_1$ where $S_1=J_1\cup J_1'$, by induction on the almost-polished element $w^{(m+1)}$, or continuing such decomposition into factors of the form $w_0(J_i)w_0(J_i\cap J_i')w_0(J_i')$, we exactly recover the definition of almost-polished elements (Definition~\ref{def:almost-polished}).
\end{proof}

\subsection{(VT3)$\Rightarrow$(VT1)}
The following lemma is straightforward.
\begin{lemma}\label{lem:VT-preserved-under-product}
If $G_1$ and $G_2$ are two simple graphs that are vertex-transitive, then $G_1\times G_2$ is vertex-transitive.
\end{lemma}
\begin{proof}
For any two vertices $(u_1,u_2),(v_1,v_2)\in G_1\times G_2$, we want to show that there exists $f\in\aut(G_1\times G_2)$ that sends $(u_1,u_2)$ to $(v_1,v_2)$. Since $G_1$ is vertex-transitive, there exists $f_1\in \aut(G_1)$ such that $f_1(u_1)=v_1$. It is clear that $f_1\times \mathrm{id}\in\aut(G_1\times G_2)$. Then since $G_2$ is vertex-transitive, there exists $f_2\in\aut(G_2)$ such that $f_2(u_2)=v_2$. Now $(\mathrm{id}\times f_2)\circ (f_1\times\mathrm{id})(u_1,u_2)=(v_1,v_2)$. 
\end{proof}
\begin{proof}[Proof of implication (VT3)$\Rightarrow$(VT1)]
Let $w$ be almost-polished (Definition~\ref{def:almost-polished}). To show that $\Gamma(w)$ is vertex-transitive, by Proposition~\ref{prop:disjoint-support} and Lemma~\ref{lem:VT-preserved-under-product}, we can reduce to the case where $w=w_0(J)w_0(J\cap J')w_0(J')$. In fact, the elementary automorphisms (Section~\ref{sub:elementary-automorphisms}) are enough in this case.

We first show that $\big(w_0(J)w_0(J\cap J')\big)\cdot w_0(J')$ is length-additive. It suffices to show that $w_0(J)w_0(J\cap J')$ does not contain any right-descent in $J'$. This is because the simple generators in $J\cap J'$ cannot be in $D_R(w_0(J)w_0(J\cap J'))$ as they get canceled out after multiplying $w_0(J)$ by $w_0(J\cap J')$; and the simple generators in $J'\setminus J$ are not even in the support of $w_0(J)w_0(J\cap J')$. As a result, $D_R(w)\supset J'$, $w^{J'}=w_0(J)w_0(J\cap J')$ and $w_{J'}=w_0(J')$. 

For any $u\in\Gamma(w)$, $u\leq w$ so $u^{J'}\leq w^{J'}\leq w_0(J)$ and $u_{J'}\leq w_0(J')$. Analogously, $D_L(w)\supset J$. By Proposition~\ref{prop:right-left-multiplication-give-autos}, left multiplying by any element in $W(J)$ and right multiplying by any element in $W(J')$ give automorphisms. In particular, $u^{J'}\in W(J)$ and $u_{J'}\in W(J')$ so $u=u^{J'}u_{J'}$ is in the same orbit as the identity element under $\aut(\Gamma(w))$. This precisely means that $\Gamma(w)$ is vertex-transitive.
\end{proof}

This completes the proof of Theorem~\ref{thm:vertex-transitive}.

\section{Identity orbits in Bruhat graphs}
\label{sec:orbits}
In this section we describe a more precise version of Conjecture~\ref{conj:orbit-is-interval}, taking into account the automorphisms described in Section~\ref{sec:automorphisms} and the classification of vertex-transitive Bruhat graphs given in Section~\ref{sec:vertex-transitive}. 

In light of Proposition~\ref{prop:disjoint-support}, it is sufficient to consider permutations $w \in \mathfrak{S}_n$ which have full support and do not admit a disjoint support decomposition; we call such permutations \emph{Bruhat irreducible}.

\begin{defin}
\label{def:almost-reducible}
A Bruhat irreducible permutation $w\in \mathfrak{S}_n$ is \emph{almost reducible} at $(J,i)$ if $w=w^Jw_J$ is a BP-decomposition with $\supp(w^J) \cap J = \{s_i\}$ and $s_i\notin D_L(w),D_R(w)$.
\end{defin}

\begin{prop}
If a Bruhat irreducible $w\in \mathfrak{S}_n$ is almost reducible at $(J,i)$, then $J=\{s_1,\ldots,s_i\}$ or $\{s_i,\ldots,s_{n-1}\}$.
\end{prop}
\begin{proof}
Let $J=J^{(1)}\sqcup\cdots\sqcup J^{(k)}$ be a decomposition of $J$ into connected components of the Dynkin diagram such that $s_i\in J^{(1)}$. If $k\geq2$, then the parabolic decomposition of $w$ with respect to $J^{(k)}$ contradicts $w$ being Bruhat irreducible, so $k=1$ and $J$ is a connected interval. Similarly, if $\supp(w^J)$ has a connected component not adjacent to $i$, $w$ cannot be Bruhat irreducible. Moreover, $J\neq\{s_i\}$ since $s_i\notin D_R(w)$. 
\end{proof}

Note that a Bruhat irreducible $w\in \mathfrak{S}_n$ is almost reducible at $(\{s_i,\ldots,s_{n-1}\},i)$ if and only if $w^{-1}$ is almost reducible at $(\{s_1,\ldots,s_i\},i)$.

\begin{defin}
A Bruhat irreducible permutation $w\in \mathfrak{S}_n$ is \emph{right-almost-reducible} at $i$ if it is almost reducible at $(\{s_i,\ldots,s_{n-1}\},i)$ and is \emph{left-almost-reducible} at $i$ if it is almost reducible at $(\{s_1,\ldots,s_{i}\},i)$.
\end{defin}

\begin{prop}\label{prop:almost-sep-supp}
If $w\in \mathfrak{S}_n$ is right-almost-reducible at $i$, then
\begin{enumerate}
\item $\max\{w(1),\ldots,w(i-1)\}=i+1$,
\item $w(i)>i+1$, and
\item the elements of $\{1,\ldots,i+1\}\setminus\{w(1),\ldots,w(i-1)\}$ appear out of order in $w$.
\end{enumerate}
\end{prop}
\begin{proof}
Consider the permutation $w_J$, which fixes $1,\ldots,i-1$, in one-line notation. By definition, $s_i\in D_L(w_J)$ meaning that $i+1$ appears before $i$ in $w_J$. Since $w$ is Bruhat irreducible, $i\in\supp(s_iw_J)$ so $s_iw_J(i)\neq i$ and thus $w_J(i)>i+1$. Now $w^J$ permutes the values $1,2,\ldots,i+1$ of $w_J$. This means that $w(i)=w_J(i)>i+1$. We clearly have $w(1),\ldots,w(i-1)\in\{1,\ldots,i+1\}$. Since $i\in\supp(w^J)$, we necessarily have $i+1$ among in $w(1),\ldots,w(i-1)$. The last item follows because $i+1$ appears before $i$ in $w_J$. 
\end{proof}

\begin{cor}
\label{cor:i-commute-with-descent}
If $w\in \mathfrak{S}_n$ is almost reducible at $(J,i)$, then $s_i$ commutes with the elements of $D_L(w)\cap D_R(w)$.
\end{cor}
\begin{proof}
Assume without loss of generality that $J=\{s_i,\ldots,s_{n-1}\}$. By Proposition~\ref{prop:almost-sep-supp}(1), $i+1$ appears before $i+2$ in $w$, so $s_{i+1}\notin D_L(w)$. By Proposition~\ref{prop:almost-sep-supp} (1) and (2), $w(i-1)\leq i+1<w(i)$ so $s_{i-1}\notin D_R(w)$. 
\end{proof}

\begin{cor}
\label{cor:i-commute-with-j}
If $w\in \mathfrak{S}_n$ is right-almost-reducible at $i$ and left-almost-reducible at $j$, then $i\neq j$ and $s_is_j=s_js_i$.
\end{cor}
\begin{proof}
We first show that $i\neq j$. Assume the opposite that both $w$ and $w^{-1}$ are right-almost-reducible at $i$. By condition (3) of Proposition~\ref{prop:almost-sep-supp}, assume $i+1>w(a)>w(b)$ where $i<a<b$, since $w(i)>i+1$. By condition (1) on $w^{-1}$, we know that $w(i+1)\leq i-1$ so $i+1$ is one of $a,b$ and it has to be $a$ because $a<b$. However, $w(b)<w(a)\leq i-1$ but $w^{-1}(w(b))=b>a=i+1$, contradicting condition (1) on $w^{-1}$.

To show that $s_i$ and $s_j$ commute, we can assume to the contrary that $j=i-1$, since if $j=i+1$, we may consider the same problem on $w^{-1}$. Now $w$ is right-almost-reducible at $i$ and left-almost-reducible at $i-1$ so $w^{-1}$ is right-almost-reducible at $i-1$. By condition (2) of $w$, $w(i)>i+1$ but by condition (1) of $w^{-1}$, $w(i)<i$, a contradiction.
\end{proof}

\begin{defin}
For a Bruhat irreducible permutation $w \in \mathfrak{S}_n$, let 
\[
\{i_1<\cdots<i_k\} = \{i \mid \text{$w$ is right-almost-reducible at $i$}\}
\]
and define $A_R(w):=s_{i_1}\cdots s_{i_k}$. Similarly, let $\{j_1<\cdots<j_t\}$ be the set of $j$ at which $w$ is left-almost-reducible and define $A_L(w):=s_{j_t}\cdots s_{j_1}.$
\end{defin}

\begin{cor}
Let $w$ be Bruhat irreducible. Then the following three elements commute pairwise: \[A_R(w),A_L(w),w_0(D_L(w)\cap D_R(w)).\]
\end{cor}

The following is a strengthened version of Conjecture~\ref{conj:orbit-is-interval}.

\begin{conj}
\label{conj:strengthened-orbit-conjecture}
Let $w \in \mathfrak{S}_n$ be Bruhat irreducible and let $\mc{O}$ denote the orbit of $e$ under graph automorphisms of $\Gamma(w)$. Define
\[
v(w) \coloneqq w_0(D_L(w))\cdot A_R(w) \cdot w_0(D_L(w)\cap D_R(w)) \cdot A_L(w)\cdot w_0(D_R(w)),
\]
then $\mc{O}=[e,v(w)]$. 
\end{conj}

\begin{prop}
\label{prop:confirm-conj-for-vt}
Let $w \in \mathfrak{S}_n$ be Bruhat irreducible and such that $\Gamma(w)$ is vertex-transitive, then $v(w)=w$, so Conjecture~\ref{conj:strengthened-orbit-conjecture} holds in this case.
\end{prop}
\begin{proof}
If $w$ is right-almost-reducible at $i$, then Proposition~\ref{prop:almost-sep-supp} and Definition~\ref{def:almost-reducible} imply that the values $i, i+1, w(i), a, b$ appear from left to right in the one-line notation for $w$ and form an occurrence of the pattern $34521$, where $a,b$ are the smallest two elements of $\{1,\ldots,i+1\} \setminus \{w(1),\ldots,w(i-1)\}$. This is impossible by Theorem~\ref{thm:vertex-transitive} since $\Gamma(w)$ is assumed to be vertex transitive. Similarly, if $w$ were left-almost-reducible at $j$, then $w$ would contain an occurrence of the pattern $54123$, again violating Theorem~\ref{thm:vertex-transitive}. Thus $A_R(w)=A_L(w)=e$, and $v(w)=w_0(D_L(w))\cdot w_0(D_L(w)\cap D_R(w)) w_0(D_R(w))$ which is the expression for $w$ as a Bruhat irreducible almost-polished element. Since $\Gamma(w)$ is vertex-transitive, we have $\mc{O}=[e,w]=[e,v(w)]$.
\end{proof}

The following proposition shows that the element $v(w)$ is indeed in the identity orbit of $\Gamma(w)$. An automorphism of $\Gamma(w)$ sending $e$ to $v(w)$ may be obtained by composing various left, right, and middle multiplication automorphisms (see Section~\ref{sec:automorphisms}).

\begin{prop}
\label{prop:v-is-in-orbit}
Let $w \in \mathfrak{S}_n$ be Bruhat irreducible and let $\mc{O}$ be the orbit of $e$ under graph automorphisms of $\Gamma(w)$, then $v(w) \in \mc{O}$.
\end{prop}
\begin{proof}
By Proposition~\ref{prop:right-left-multiplication-give-autos} we may compose left (or right) multiplication automorphisms to send $e$ to $w_0(D_L(w) \cap D_R(w))$. Then, for each $i$ such that $w$ is right-almost-reducible at $i$, by Proposition~\ref{prop:middle-mult-is-automorphism} we may apply the automorphism of middle multiplication by $s_i$, doing so in the order $i_k > i_{k-1} > \cdots > i_1$, and similarly for indices $j_1 < \cdots < j_t$ at which $w$ is left-almost-reducible. Since all of these $s_i$ and $s_j$ commute with each other and with the simple generators in $D_L(w) \cap D_R(w)$ by Corollaries~\ref{cor:i-commute-with-descent} and \ref{cor:i-commute-with-j}, the middle multiplication is equivalent to left multiplication at each stage, and the resulting product is equal to $A_R(w)w_0(D_L(w) \cap D_R(w))A_L(w)$. Finally, applying left and right multiplication by $w_0(D_L(w))$ and $w_0(D_R(w))$ respectively, we obtain an automorphism sending $e$ to $v(w)$.
\end{proof}

\section{Special matchings and Bruhat automorphisms}
\label{sec:special-matchings}
\subsection{The connection to special matchings}
The proof of Proposition~\ref{prop:v-is-in-orbit} and Conjecture~\ref{conj:strengthened-orbit-conjecture} would together imply that the left, right, and middle multiplication automorphisms suffice to determine the identity orbit of $\Gamma(w)$ under graph automorphisms. Left or right multiplication by a descent of $w$ determines a special matching of the Hasse diagram $H([e,w])$ of $[e,w]$, essentially by construction; in Proposition~\ref{prop:middle-mult-is-special-matching} below, we observe that the same is true for middle multiplication. This result could also be obtained from the classification of special matchings of lower Bruhat intervals given in \cite{SM-first-lower-classification, SM-lower-classification}.

\begin{prop}
\label{prop:middle-mult-is-special-matching}
Suppose $w=w^Jw_J$ is a BP-decomposition of $w$ and in addition we have $\supp(w^J) \cap \supp(w_J) =\{s\}$, then the middle multiplication map 
\[
\phi: x \mapsto x^J s x_J, 
\]
is a special matching of $[e,w]$.
\end{prop}
\begin{proof}
Clearly $\phi$ gives a perfect matching on $H([e,w])$, so it suffices to check that for $x \lessdot y \in [e,w]$ we have $\phi(x)=y$ or $\phi(x)<\phi(y)$. By Proposition~\ref{prop:middle-mult-is-automorphism}, $\phi$ is an automorphism of $\Gamma(w)$, so $\phi(x),\phi(y)$ differ by multiplication by a reflection, and in particular either $\phi(x)<\phi(y)$ or $\phi(y)<\phi(x)$. In the first case we are done, so suppose $\phi(y)<\phi(x)$. By Proposition~\ref{prop:monotonicity-of-projection}, and by the construction of middle multiplication, we have:
\[
y^J = \phi(y)^J \leq \phi(x)^J = x^J.
\]
Since $x^J \leq y^J$, we must in fact have $x^J=y^J$.

Now, $\phi(y)_J=sy_J$ and $\phi(x)_J=sx_J$; since $\phi(y)<\phi(x)$ and $\phi(y)^J=\phi(x)^J$, it must be that $sy_J<sx_J$. We also know $x_J < y_J$, so by the Lifting Property we must have $sy_J=x_J$ and thus $\phi(y)=x$. 
\end{proof}

The fact that left, right, and middle multiplication determine special matchings, and the conjectural fact that they determine at least the identity orbit structure of Bruhat graphs, suggest a close connection between special matchings and automorphisms of Bruhat graphs. This connection is made explicit in Theorem~\ref{thm:SM-equals-automorphism-classical-type}, whose proof occupies the remainder of this section, and in Conjecture~\ref{conj:SM-equals-automorphism-general}.

\subsection{Special matching are Bruhat automorphisms}
Theorem~\ref{thm:SM-implies-auto} provides one direction of Theorem~\ref{thm:SM-equals-automorphism-classical-type} and Conjecture~\ref{conj:SM-equals-automorphism-general} for arbitrary Coxeter groups. This implies that special matchings on Bruhat intervals, although defined by a local condition (that is, a condition on cover relations), respect the global structure of Bruhat graphs.

\begin{remark}
We thank Mario Marietti for alerting us to the fact that, although it is stated only for lower intervals $[e,v]$, the proof of Theorem 10.3 in \cite{SM-advances} also applies to general intervals $[u,v]$ and yields Theorem~\ref{thm:SM-implies-auto}. The proof given there is very similar to our proof, which we include below for the reader's convenience.
\end{remark}

We first prove Lemma~\ref{lem:undirected-hexagon}, an extension to $\Gamma$ of a property of $\widehat{\Gamma}$ given in the proof of Proposition 3.3 of \cite{dyer-bruhat-graph}. 

\begin{lemma}
\label{lem:undirected-hexagon}
Let $u \leq v$ be elements of a Coxeter group $W$ and suppose that there exist elements $x_1,\ldots,x_6 \in [u,v]$ such that $\overline{x_1x_2}, \overline{x_1x_3}, \overline{x_2x_4}, \overline{x_2x_5}, \overline{x_3x_4}, \overline{x_3x_5}, \overline{x_4x_6}, \overline{x_5x_6}$ are edges in $\Gamma(u,v)$. Then $\overline{x_1x_6}$ is an edge in $\Gamma(u,v)$.
\end{lemma}
\begin{proof}
Let $t_1,\ldots,t_8$ be the reflections corresponding to the known edges given in the statement of the lemma. The same argument as in Proposition 3.3 of \cite{dyer-bruhat-graph} implies that $W'=\langle t_1, \ldots, t_8 \rangle$ is a dihedral reflection subgroup of $W$. By Theorem 1.4 of \cite{dyer-bruhat-graph}, the Bruhat graph of $W'$ agrees with the induced subgraph $\Gamma|_{W'}$, so it suffices to check the lemma in the case $W$ is dihedral. In this case $\Gamma$ is easy to describe: we have $\overline{xy}$ if and only if $\ell(y)-\ell(x)$ is odd. By assumption, each of $\ell(x_2)-\ell(x_1), \ell(x_4)-\ell(x_2),$ and $\ell(x_6)-\ell(x_4)$ is odd, thus $\ell(x_6)-\ell(x_1)$ is also odd, and $\overline{x_1x_6}$ is an edge.
\end{proof}

\begin{theorem}[cf. Theorem 10.3 of \cite{SM-advances}]
\label{thm:SM-implies-auto}
Let $u \leq v$ be elements of a Coxeter group $W$. Any special matching $M$ of the Hasse diagram $H([u,v])$ is an automorphism of $\Gamma(u,v)$. 
\end{theorem}
\begin{proof}
Let $M$ be a special matching of $H([u,v])$. We will prove by induction on $k$ that if $\overline{xy}$ is an edge of $\Gamma(u,v)$ with $\ell(y)-\ell(x)=k$, then $\overline{M(x)M(y)}$ is also an edge.

Consider first the case $k=1$, so $x \lessdot y$. If $M(x)=y$ or $|\ell(M(y))-\ell(M(x))|=1$, we are done by the defining property of special matchings, so assume that $M(x)\lessdot x$ and $y \lessdot M(y)$. Since all height-two intervals in Bruhat order are diamonds, there exist elements $M(x) \lessdot x' \neq x$ and $y \neq y' \lessdot M(y)$. Since $M$ is a special matching, we must have $M(x) \lessdot M(y') \lessdot y$. Again applying the diamond property to $[M(x),y]$, we conclude that $M(y')=x'$, so in particular $x' \lessdot y'$. Now, the elements $M(x),x,x',y,y',M(y)$ form a subgraph of $\Gamma$ of the type described in Lemma~\ref{lem:undirected-hexagon}, and so $\overline{M(x)M(y)}$ is an edge by the lemma.

Suppose now that $k>1$. By the proof of Proposition 3.3 from \cite{dyer-bruhat-graph}, we know that there exist elements $x_2,x_3,x_4,x_5$ with directed edges $x=x_1 \to x_2,x_3; x_2\to x_4,x_5; x_3 \to x_4,x_5; x_4,x_5 \to x_6=y$ in $\widehat{\Gamma}(u,v)$. By induction, we know that $\overline{M(x_i)M(x_j)}$ is an edge in $\Gamma(u,v)$ for each of these edges $x_i \to x_j$. Thus $M(x_1),\ldots,M(x_6)$ form a subgraph of $\Gamma$ of the type described in Lemma~\ref{lem:undirected-hexagon}, and so again we conclude that $\overline{M(x)M(y)}$ is an edge.
\end{proof}

\subsection{Bruhat automorphisms are special matchings}

In Theorem~\ref{thm:auto-implies-SM} below we give a converse to Theorem~\ref{thm:SM-implies-auto} for certain Coxeter groups. Theorem~\ref{thm:SM-implies-auto} and Theorem~\ref{thm:auto-implies-SM} together imply Theorem~\ref{thm:SM-equals-automorphism-classical-type}.

\begin{theorem}
\label{thm:auto-implies-SM}
Let $u \leq v$ be elements of a Coxeter group $W$ which is right-angled or a symmetric group, then any perfect matching of $H([u,v])$ which is an automorphism of $\Gamma(u,v)$ is a special matching.
\end{theorem}

The proof of Theorem~\ref{thm:auto-implies-SM} relies on the following structural property of Bruhat order, Lemma~\ref{lem:butterfly}, whose proof is contained in Section~\ref{sec:butterflies}. 

\begin{defin}\label{def:butterfly}
We say that elements $x_1,x_2,y_1,y_2$ of a Coxeter group $W$ form a \emph{butterfly} if $x_1 \lessdot y_1,y_2$ and $x_2 \lessdot y_1,y_2$.
\end{defin}
The butterfly structures are essential to the analysis of Bruhat automorphisms and special matchings, and are of interest on their own. We will explore more about butterflies in Section~\ref{sec:butterflies}.

\begin{lemma}
\label{lem:butterfly}
Let $W$ be a Coxeter group which is right-angled or the symmetric group, let $u \leq v$, and suppose that $x_1,x_2,y_1,y_2 \in [u,v]$ form a butterfly. Then there is an element $z \in [u,v]$ with $y_1,y_2 \lessdot z$.  
\end{lemma}

\begin{proof}[Proof of Theorem~\ref{thm:auto-implies-SM}]
Let $u \leq v$ be elements of a Coxeter group $W$ which is right-angled or the symmetric group, and let $M$ be a perfect matching of $H([u,v])$ which is an automorphism of $\Gamma(u,v)$. Suppose that $M$ is not a special matching; since $M$ is a $\Gamma(u,v)$-automorphism, the violation of the special matching property must consist of elements $x \lessdot y$ with $M(y) \lessdot M(x)$. Choose $x,y$ so that $y$ has maximal length among all such violations in $[u,v]$. 

Now, note that $x,M(y),y,M(x)$ form a butterfly, so by Lemma~\ref{lem:butterfly} there exists an element $z \in [u,v]$ with $y,M(x) \lessdot z$. We must have $M(z)>z$, for otherwise each of $y,M(x),$ and $M(z)$ would each cover both $x$ and $M(y)$, but this substructure cannot occur in Bruhat order of a Coxeter group (see Theorem 3.2 of \cite{SM-advances}). Since height-two intervals in Bruhat order are diamonds (see Chapter 2 of \cite{bjorner-brenti}), there exists an element $w\neq z$ with $y \lessdot w \lessdot M(z)$.

Suppose that $M(w)<w$, then since $M$ is an automorphism of the Bruhat graph we must have $M(w) \lessdot z$ and $M(y) \lessdot M(w)$. Now, since $y \lessdot z$, we know $M(y) \to M(z)$ in $\widehat{\Gamma}(u,v)$, but the height-three interval $[M(y),M(z)]$ contains at least three elements---$y,M(w),$ and $M(x)$ at height one, contradicting Proposition 3.3 of \cite{dyer-bruhat-graph}.

We conclude that $w\lessdot M(w)$. However this too is a contradiction, since $w \lessdot M(z)$ is a violation of the special matching condition with $\ell(M(z))>\ell(y)$. Thus $M$ must be a special matching.
\end{proof}

We conjecture that a slight weakening of Lemma~\ref{lem:butterfly} holds for arbitrary Coxeter groups, and this would imply the same for Theorem~\ref{thm:auto-implies-SM}, and thus resolve Conjecture~\ref{conj:SM-equals-automorphism-general}.

\begin{conj}
\label{conj:general-butterfly}
Let $W$ be any Coxeter group, let $u \leq v \in W$, and suppose that the elements $x_1,x_2,y_1,y_2 \in [u,v]$ form a butterfly. Then there is an element $z \in [u,v]$ with $y_1,y_2 \lessdot z$ or with $z \lessdot x_1,x_2$.
\end{conj}

\begin{remark}
The weakening of Lemma~\ref{lem:butterfly} conjectured for general Coxeter groups in Conjecture~\ref{conj:general-butterfly} is necessary even for finite Coxeter groups. For example, there exists a butterfly in the finite Coxeter group of type $F_4$ which has a lower bound $z \lessdot x_1,x_2$ but no upper bound $y_1,y_2 \lessdot z'$.
\end{remark}

\section{Covers of butterflies in Bruhat order}\label{sec:butterflies}

\subsection{Butterflies in finite Weyl groups}
Recall that a butterfly consists of four elements with $x_1,x_2\lessdot y_1,y_2$ in Bruhat order. We first do some general analysis on butterflies in finite Weyl groups.

Consider the transpositions $t_{ab}=x_a^{-1}y_b$ where $a,b\in\{1,2\}$. Then $t_{11}t_{21}=t_{12}t_{22}$. By Lemma 3.1 of \cite{dyer-bruhat-graph}, $W'=\langle t_{11},t_{12},t_{21},t_{22}\rangle$ is a reflection subgroup of $W$. We say that this butterfly $x_1,x_2,y_1,y_2$ is \emph{of type} $A_2$ if $W'$ is isomorphic to the Coxeter group of type $A_2$, and same with type $B_2$ and $G_2$. By Theorem 1.4 of \cite{dyer-bruhat-graph}, the subposet (and the directed subgraph) on $x_1W'$ of $W$ is isomorphic to that of a rank $2$ Coxeter group of this type. This means that $W'$ cannot be of type $A_1\times A_1$, because a butterfly cannot be embedded in the Bruhat order of the type $A_1\times A_1$ Coxeter group, which looks like a diamond \begin{tikzpicture}[scale=0.15]
\draw(-1,0)--(0,1)--(1,0)--(0,-1)--(-1,0);
\end{tikzpicture}. In finite classical types, only type $A_2$ and type $B_2$ butterflies exist. 

Moreover, if this butterfly is of type $A_2$, we must have that $x_1W'$ consists of $u<x_1,x_2\lessdot y_1,y_2<z$ with edges from $u$ to $x_1,x_2$ and edges from $y_1,y_2$ to $z$ in the Bruhat graph. Similarly in the case of type $B_2$, we must have that $x_1W'$ consists of $u<a_1,a_2<b_1,b_2<c_1,c_2<z$ with edges in the Bruhat graph $u\rightarrow a_1,a_2\rightarrow b_1,b_2\rightarrow c_1,c_2\rightarrow z$ where $\{x_1,x_2\}=\{a_1,a_2\}$, $\{y_1,y_2\}=\{b_1,b_2\}$ or $\{x_1,x_2\}=\{b_1,b_2\}$, $\{y_1,y_2\}=\{c_1,c_2\}$.

Let $\Phi\subset E$ be a root system for the finite Weyl group $W$, where $E$ is the ambient vector space with an inner product $\langle-,-\rangle$, with a chosen set of positive roots $\Phi^+$ and simple roots $\Delta$. For $\alpha\in\Phi^+$, write $s_{\alpha}\in T$ be the reflection across $\alpha$. Recall that the inversion set is $\Inv_R(w)=\{\alpha\in\Phi^+\:|\: w\alpha\in\Phi^+\}$ so that $T_R(w)=\{s_{\alpha}\:|\: \alpha\in \Inv_R(w)\}$.

We say that a butterfly in a finite Weyl group $W$ \emph{is generated by} $\alpha,\beta\in\Phi^+$ if $s_{\alpha}$ and $s_{\beta}$ generate the subgroup $W'$ and that $\alpha$ and $\beta$ form a set of simple roots in the root subsystem $\Phi$ restricted to the $2$-dimensional vector space spanned by $\alpha$ and $\beta$. Note that the generators $\{\alpha,\beta\}$ of a butterfly is fixed. To analyze butterflies, we start with some a simple lemma on Bruhat covers.

\begin{lemma}\label{lem:weyl-cover}
In a finite Weyl group, $w\lessdot ws_{\alpha}$ if and only if $\alpha\notin\Inv_R(w)$ and there does not exist $\beta_1,\beta_2\in\Inv_R(w)$ such that $\beta_2=-s_{\alpha}\beta_1$. Moreover, if $w\lessdot ws_{\alpha}$ and $\beta\in\Phi^+$ satisfies $s_{\alpha}\beta\in\Phi^-$, then $\beta\in\Inv_R(w)$ if and only if $\beta\in\Inv_R(ws_{\alpha})$.
\end{lemma}
\begin{proof}
Consider the following partition of $\Phi^+$ with respect to $\alpha\in\Phi^+$:
\begin{enumerate}
\item the root $\alpha$ itself;
\item roots $\gamma\in\Phi^+$ such that $s_{\alpha}\gamma=\gamma$;
\item roots $\gamma\in\Phi^+$ such that $s_{\alpha}\gamma\in\Phi^+$ but $s_{\alpha}\gamma\neq\gamma$;
\item roots $\beta\in\Phi^+$ such that $s_{\alpha}\beta\in\Phi^-$.
\end{enumerate}
We pair up roots in (3) by $(\gamma,s_{\alpha}\gamma)$ and pair up roots in (4) by $(\beta,-s_{\alpha}\beta)$.

Assume $\alpha\notin\Inv_R(w)$ and compare $\Inv_R(w)$ with $\Inv_R(ws_{\alpha})$. First, $\alpha\notin\Inv_R(w)$ and $\alpha\in\Inv_R(ws_{\alpha})$, and for each root $\gamma$ in (2), we have $\gamma\in\Inv_R(w)\Leftrightarrow\gamma\in\Inv_R(ws_{\alpha})$. Similarly, for $\gamma$ in (3), we have $\gamma\in\Inv_R(w)\Leftrightarrow s_{\alpha}\gamma\in\Inv_R(ws_{\alpha})$. Thus, roots in (1), (2) and (3) each contribute $1$ to the quantity $|\Inv_R(ws_{\alpha})|-|\Inv_R(w)|$. We now examine (4).

Let $\beta$ be a root in (4) and also write $\beta_1=\beta$ and $\beta_2=-s_{\alpha}\beta$. We have $s_{\alpha}\beta=\beta-c\alpha$, where $c=2\langle\alpha,\beta\rangle/\langle\alpha,\alpha\rangle\in\mathbb{Q}_{>0}$. As $\alpha\notin\Inv_R(w)$, we know $w\alpha\in\Phi^+$. So $w(\beta_1+\beta_2)=cw\alpha>0$, meaning that at most one of $\beta_1,\beta_2$ belong to $\Inv_R(w)$. Moreover, $\beta_1\notin\Inv_R(w)$ if and only if $\beta_2\in\Inv_R(ws_{\alpha})$. This means that if none of $\beta_1,\beta_2$ belong to $\Inv_R(w)$, then both belong to $\Inv_R(ws_{\alpha})$, contributing $2$ to $|\Inv_R(ws_{\alpha})|-|\Inv_R(w)|$; and if one of them belongs to $\Inv_R(w)$, then the same one belongs to $\Inv_R(ws_{\alpha})$.

Note that $w\lessdot ws_{\alpha}$ is equivalent to $|\Inv_R(ws_{\alpha})|-|\Inv_R(w)|=1$. Considering the above contributions from each category of roots, we obtain the desired result. 
\end{proof}
Note that Lemma~\ref{lem:weyl-cover} is also equivalent to saying that $ws_{\alpha}\lessdot w$ if and only if $\alpha\in\Inv_R(w)$ and there does not exist $\beta_1,\beta_2\in\Inv_R(w)$ such that $\beta_2=-s_{\alpha}\beta_1$. In simply-laced types, assume $\langle\alpha,\alpha\rangle=2$ for all roots $\alpha\in\Phi$, then all inner products between different positive roots take on values in $\{0,1,-1\}$, and the condition $\beta_2=-s_{\alpha}\beta_1$ is equivalent to $\beta_1+\beta_2=\alpha$.

\begin{lemma}\label{lem:simply-laced-join}
Let $W$ be a finite Weyl group of simply-laced types, and let $x_1,x_2\lessdot y_1,y_2$ form a butterfly. Then there exists $u\lessdot x_1,x_2$ and $z\gtrdot y_1,y_2$ in $W$.
\end{lemma}
\begin{proof}
Since $W$ is simply-laced, this butterfly can only be of type $A_2$. Let $u<x_1,x_2\lessdot y_1,y_2\lessdot z$ be this type $A_2$ subposet. We will show that $u$ is covered by $x_1$ and $x_2$. By taking the dual statement, we will have $y_1,y_2\lessdot z$ as well.

Let this butterfly be generated by $\alpha,\beta\in\Phi^+$ and $us_{\alpha}=x_1$, $us_{\beta}=x_2$, $us_{\alpha}s_{\beta}=y_2$, $us_{\beta}s_{\alpha}=y_1$. We have $\langle\alpha,\beta\rangle=-1$, $\alpha+\beta=s_{\alpha}\beta=s_{\beta}\alpha\in\Phi^+$, and we also know that in this $A_2$, $\alpha\in\Inv_R(x_1),\Inv_R(y_2),\Inv_R(z)$, $\beta\in\Inv_R(x_2),\Inv_R(y_1),\Inv_R(z)$, $\alpha+\beta\in\Inv_R(y_1),\Inv_R(y_2),\Inv_R(z)$. If $x_2$ does not cover $u$ (or equivalently, $x_1$ does not cover $u$), by Lemma~\ref{lem:weyl-cover}, there exists $\gamma_1,\gamma_2\in\Inv_R(x_2)$ such that $\gamma_1+\gamma_2=\beta$, or equivalently, $\gamma_2=-s_{\beta}\gamma$. We have
\[\langle\gamma_1,\alpha+\beta\rangle+\langle\gamma_1,\alpha+\beta\rangle=\langle\beta,\alpha+\beta\rangle=1.\]
Since all inner products between different positive roots lie in $\{0,1,-1\}$, we can without loss of generality assume that $\langle\gamma_1,\alpha+\beta\rangle=0$ and $\langle\gamma_2,\alpha+\beta\rangle=1$. Now $s_{\alpha+\beta}\gamma_1=\gamma_1$. Since $\gamma_1\in\Inv_R(x_2)$ and $y_1=s_{\alpha+\beta}x_2$, we have $\gamma_1\in\Inv_R(y_1)$. Moreover, since $\langle\gamma_2,\alpha+\beta\rangle=1$, $s_{\alpha+\beta}\gamma_2=\gamma_2-(\alpha+\beta)=-(\alpha+\gamma_1)\in\Phi^-$. By Lemma~\ref{lem:weyl-cover}, as $x_2\lessdot y_1$, $\gamma_2\in\Inv_R(y_1)$. But $\gamma_1,\gamma_2\in\Inv_R(y_1)$ with $s_{\beta}\gamma_2=-\gamma_1$, contradicting $y_1\gtrdot y_1s_{\beta}=x_1$.
\end{proof}

\subsection{Butterflies in the symmetric group}
For $w\in \mathfrak{S}_n$, define its \emph{rank-matrix} to be $w[i,j]=|\{a\in[i]\:|\: w(a)\geq j\}|$, which can be viewed as the number of dots weakly in the bottom left corner in the permutation matrix of $w$, for all $i,j\in[n]$. See Figure~\ref{fig:rank-matrix-type-A}.
\begin{figure}[h!]
\centering
\begin{tikzpicture}[scale=0.4]
\draw(0,0)--(5,0)--(5,-5)--(0,-5)--(0,0);
\draw[dashed](0,-1)--(5,-1);
\draw[dashed](0,-2)--(5,-2);
\draw[dashed](0,-3)--(5,-3);
\draw[dashed](0,-4)--(5,-4);
\draw[dashed](1,0)--(1,-5);
\draw[dashed](2,0)--(2,-5);
\draw[dashed](3,0)--(3,-5);
\draw[dashed](4,0)--(4,-5);
\node at (0.5,0.5) {$1$};
\node at (1.5,0.5) {$2$};
\node at (2.5,0.5) {$3$};
\node at (3.5,0.5) {$4$};
\node at (4.5,0.5) {$5$};
\node at (-0.5,-0.5) {$1$};
\node at (-0.5,-1.5) {$2$};
\node at (-0.5,-2.5) {$3$};
\node at (-0.5,-3.5) {$4$};
\node at (-0.5,-4.5) {$5$};
\node at (0.5,-2.5) {$\bullet$};
\node at (1.5,-4.5) {$\bullet$};
\node at (2.5,-0.5) {$\bullet$};
\node at (3.5,-3.5) {$\bullet$};
\node at (4.5,-1.5) {$\bullet$};
\end{tikzpicture}
\qquad
\begin{tikzpicture}[scale=0.4]
\draw(0,0)--(5,0)--(5,-5)--(0,-5)--(0,0);
\draw[dashed](0,-1)--(5,-1);
\draw[dashed](0,-2)--(5,-2);
\draw[dashed](0,-3)--(5,-3);
\draw[dashed](0,-4)--(5,-4);
\draw[dashed](1,0)--(1,-5);
\draw[dashed](2,0)--(2,-5);
\draw[dashed](3,0)--(3,-5);
\draw[dashed](4,0)--(4,-5);
\node at (0.5,0.5) {$1$};
\node at (1.5,0.5) {$2$};
\node at (2.5,0.5) {$3$};
\node at (3.5,0.5) {$4$};
\node at (4.5,0.5) {$5$};
\node at (-0.5,-0.5) {$1$};
\node at (-0.5,-1.5) {$2$};
\node at (-0.5,-2.5) {$3$};
\node at (-0.5,-3.5) {$4$};
\node at (-0.5,-4.5) {$5$};

\node at (0.5,-0.5) {$1$};
\node at (0.5,-1.5) {$1$};
\node at (0.5,-2.5) {$1$};
\node at (0.5,-3.5) {$0$};
\node at (0.5,-4.5) {$0$};
\node at (1.5,-0.5) {$2$};
\node at (1.5,-1.5) {$2$};
\node at (1.5,-2.5) {$2$};
\node at (1.5,-3.5) {$1$};
\node at (1.5,-4.5) {$1$};
\node at (2.5,-0.5) {$3$};
\node at (2.5,-1.5) {$2$};
\node at (2.5,-2.5) {$2$};
\node at (2.5,-3.5) {$1$};
\node at (2.5,-4.5) {$1$};
\node at (3.5,-0.5) {$4$};
\node at (3.5,-1.5) {$3$};
\node at (3.5,-2.5) {$3$};
\node at (3.5,-3.5) {$2$};
\node at (3.5,-4.5) {$1$};
\node at (4.5,-0.5) {$5$};
\node at (4.5,-1.5) {$4$};
\node at (4.5,-2.5) {$3$};
\node at (4.5,-3.5) {$2$};
\node at (4.5,-4.5) {$1$};
\end{tikzpicture}
\caption{The rank-matrix (right) for the permutation $w=35142$ (left).}
\label{fig:rank-matrix-type-A}
\end{figure}
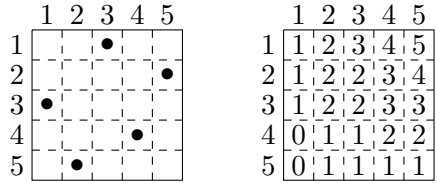

The following lemma is immediate from observation.
\begin{lemma}\label{lem:rank-change-outside-box}
Let $w\in \mathfrak{S}_n$ and $w'=w\cdot (i,j)$. Then $w[a,b]=w'[a,b]$ for the coordinate $(a,b)$ outside, or on the top or right boundary, of the rectangle formed by $(i,w(i))$ and $(j,w(j))$.
\end{lemma}

The following characterization of the strong order in $\mathfrak{S}_n$ is well-known.
\begin{lemma}[Theorem 2.1.5 of \cite{bjorner-brenti}]\label{lem:strong-order-type-A}
For $w,v\in \mathfrak{S}_n$, $w\leq v$ if and only if $w[i,j]\leq v[i,j]$ for all $i,j\in[n]$.
\end{lemma}

We are now ready to prove Lemma~\ref{lem:butterfly} for the symmetric group.

\begin{proof}[Proof of Lemma~\ref{lem:butterfly} in the case of type $A_{n-1}$]
Let $x_1,x_2\lessdot y_1,y_2$ be a butterfly in $[u,v]$. By Lemma~\ref{lem:simply-laced-join}, there exists $z\gtrdot y_1,y_2$ so it suffices to show that $z\leq v$. Suppose that $y_1=z\cdot (i,j)$, $y_2=z\cdot (j,k)$ with $i<j<k$ and $z(i)>z(j)>z(k)$. There is no dot strictly inside the rectangle formed by $(i,z(i))$ and $(j,z(k))$ in the permutation matrix of $z$, because otherwise $y_2$ would have inversions at some $e_p-e_i$ and $e_j-e_p$, contradicting $y_2\gtrdot y_2\cdot (i,j)$ by Lemma~\ref{lem:weyl-cover}. Likewise, we see that there are only one dot, $(j,z(j))$, in the interior rectangle formed by $(i,z(i))$ and $(k,z(k))$ in the permutation matrix of $z$.

To show that $z\leq v$, by Lemma~\ref{lem:strong-order-type-A}, it suffices to show that $z[a,b]\leq v[a,b]$ for all $a,b\in[n]$. By Lemma~\ref{lem:rank-change-outside-box}, $z[a,b]=y_1[a,b]$ if $(a,b)$ is not in the interior or on the left or bottom boundary on the rectangle formed by $(i,z(i))$ and $(j,z(j))$; $z[a,b]=y_2[a,b]$ if $(a,b)$ is not in the interior or the bottom left boundary of the rectangle $(j,z(j))$ and $(k,z(k))$. Noticing that these regions are disjoint, we have $z[a,b]=y_1[a,b]$ or $y_2[a,b]$. But $v\geq y_1,y_2$ so $v[a,b]\geq \max{y_1[a,b],y_2[a,b]}\geq z[a,b]$, which gives $z\leq v$ as desired.
\end{proof}

\subsection{Butterflies in right-angled Coxeter groups}
Throughout this section, let $W$ be a right-angled Coxeter group and let $S$ be its generating set. By definition, any two elements $s_i,s_j\in S$ either commute or have no relations. Recall a well-known result of Tits \cite{Tits-words} that says all reduced expressions of $w$ are connected by moves of the form $s_is_j \cdots = s_js_i \cdots$ with $m_{ij}$ factors on each side, and in this case, only $m_{ij}=2$ needs to be considered. In other words, any two reduced expressions of $w\in W$ are connected by commutation moves. 

\begin{lemma}\label{lem:descent-if-can-move}
Let $w=s_{i_1}\cdots s_{i_{\ell}}$ be any reduced word of $w$. Then $s\in D_L(w)$ if and only if the minimal $j$ such that $s_{i_j}=s$ commutes with $s_{i_1},\ldots,s_{i_{j-1}}$.
\end{lemma}
\begin{proof}
First, if $s$ commutes with $s_{i_1},\ldots,s_{i_{j-1}}$, then we can move it all the way to the left via commutation moves to obtain a reduced word of $w$ starting with $s$, which means $s\in D_L(w)$. On the other hand, if $s\in D_L(w)$, then we use another reduced word of $w$ that starts with $s$. Keeping track of this $s$ and applying commutation moves, we see that only $s_i$'s commuting with $s$ can ever appear on the left of this $s$, which always stays as the first appearance of $s$ in any reduced word. We are done because any two reduced words of $w$ are connected via commutation moves.
\end{proof}

\begin{lemma}
If $y$ covers two elements $x_1,x_2$ and $s\in D_L(x_1), D_L(x_2)$, then $s\in D_L(y)$.
\end{lemma}
\begin{proof}
Let $y=s_{i_1}\cdots s_{i_{\ell}}$ be a reduced expression of $y$. By the Subword Property, assume that $x_1=s_{i_1}\cdots \hat{s}_{i_a}\cdots s_{i_{\ell}}$, $x_2=s_{i_1}\cdots \hat{s}_{i_b}\cdots s_{i_{\ell}}$ with $a<b$. Let $s_{i_j}=s$ be the first appearance of $s$ in this reduced of $y$. The existence of $j$ follows from $s\in D_L(x_1),D_L(x_2)$.

Case 1.: $a\leq j<b$. The prefixes of length $j$ in $y$ and $x_2$ are the same. By Lemma~\ref{lem:descent-if-can-move}, since $s\in D_L(x_2)$, $s$ commutes with $s_{i_1},\ldots,s_{i_{j-1}}$. And by Lemma~\ref{lem:descent-if-can-move} again, $s\in D_L(y)$.

Case 2: $j\geq b$. The first appearance of $s$ in $x_1$ must be at index $j$, meaning that $s$ commutes with $s_{i_1},\ldots,s_{i_{a-1}}$ and $s_{i_{a+1}},\ldots,s_{i_{j-1}}$. The first appearance of $s$ in $x_2$ must be after index $a$, meaning that $s$ commutes with $s_{i_a}$ as well. Together, we see that $s$ commutes with all the $s_i$'s before index $j$ so $s\in D_L(y)$.
\end{proof}

We are now ready to provide a detailed analysis on the structures of butterflies in right-angled Coxeter groups. For $s,s'\in S$ that do not commute, define an element \[A^{(m)}(s,s')=ss'ss'\cdots\in W\] with $m$ copies of $s$ and $s'$ multiplied in an alternating way. 
\begin{lemma}\label{lem:butterfly-right-angle-structure}
Let $x_1,x_2\lessdot y_1,y_2$ form a butterfly in a right-angled Coxeter group $W$. Then we have the length-additive expressions: $x_1=u\cdot A^{(m)}(s,s')\cdot v$, $x_2=u\cdot A^{(m)}(s',s)\cdot v$, $\{y_1,y_2\}=\{u\cdot A^{(m+1)}(s,s')\cdot v,u\cdot A^{(m+1)}(s',s)\cdot v\}$ for some $m\geq1$, $u,v\in W$ and $s,s'\in S$ that do not commute.
\end{lemma}
\begin{proof}
Use induction on $\ell(x_1)$. If $x_1$ and $x_2$ have a common left descent $s$, then all these four elements have the same left descent $s$ and we can instead consider the butterfly $sx_1,sx_2\lessdot sy_1,sy_2$. Thus, assume that $x_1$ and $x_2$ do not have any common left descents, and similarly do not have any common right descents.

Choose a reduced word $y_1=s_{i_1}s_{i_2}\cdots s_{i_k}$ and by the Subword Property, let $x_2$ be obtained from $y_1$ by deleting $s_{i_a}$ and $x_1$ be obtained from $y_1$ by deleting $s_{i_b}$ with $a<b$. We must have $a=1$, since otherwise $s_{i_1}$ is a common descent of $x_1$ and $x_2$. Similarly $b=k$. Moreover, $s_{i_1}$ must be the unique descent of $x_1$ since any other potential descent $s_{i_c}$ will be a descent of $x_2$ by Lemma~\ref{lem:descent-if-can-move}. Similarly, write $y_2=s_{i_1}'s_{i_2}'\cdots s_{i_k}'$ then we analogously have $x_1,x_2\in\{s_{i_1}'\cdots s_{i_{k-1}}', s_{i_2}'\cdots s_{i_{k}}'\}$. If $x_1=s_{i_1}'\cdots s_{i_{k-1}}'$, then $s_{i_1}'=s_{i_1}$ since $x_1$ has a unique left descent, which means $y_2=s_{i_1}'x_2=s_{i_1}x_2=y_1$, a contradiction. Thus, we have
\[\begin{cases}
x_1=&s_{i_1}s_{i_2}\cdots s_{i_{k-1}}=s_{i_2}'s_{i_3}'\cdots s_{i_k}'\\
x_2=&s_{i_2}s_{i_3}\cdots s_{i_k}=s_{i_{1}}'s_{i_2}'\cdots s_{i_{k-1}}'
\end{cases}\]
and each one of $x_1$ and $x_2$ has a unique left descent and a unique right descents. 

Let $s=s_{i_1}$ and $s'=s_{i_1}'$, which are different. We now use induction on $p=1,\ldots,k$ to show that: $s_{i_p}s_{i_{p+1}}\cdots s_{i_k}$ has a single left descent at $s_{i_p}$, $s_{i_p}'s_{i_{p+1}}'\cdots s_{i_k}'$ has a single left descent at $s_{i_p}'$ and that $s_{i_p}=s$, $s_{i_p}'=s'$ if $p$ is odd, $s_{i_p}=s'$, $s_{i_p}'=s$ if $p$ is even. As for the base case $p=1$, we need to show that $y_1$ has a single left descent at $s$. Note that we already know that $x_1$ has a single left descent at $s$ so the possibility that $y_1=x_1s_{i_k}$ has another left descent is that $s_{i_k}$ commutes with $x_1$, which is impossible as we also know that $s_{i_k}$ cannot get pass $s_{i_{k-1}}$. As a result, $y_1$ has a single left descent at $s$ and analogously $y_2$ has a single left descent at $s'$.

For the inductive step, assume the claims are true for $p-1$. By the induction hypothesis, $s_{i_j}=s_{i_{j+1}}'$ and $s_{i_j}'=s_{i_{j+1}}$ for $j\leq p-2$. This means that we have $s_{i_{p-1}}\cdots s_{i_{k-1}}=s_{i_p}'\cdots s_{i_{k}}'$. By the induction hypothesis, $s_{i_{p-1}}\cdots s_{i_{k}}$ has a single left descent so $s_{i_{p-1}}\cdots s_{i_{k-1}}$ has a single left descent at $s_{i_{p-1}}$, which must equal $s_{i_p}'$ because of this equality. Similarly, $s_{i_p}\cdots s_{i_{k}}$ has a single left descent at $s_{i_{p-1}}'$ which also gives $s_{i_p}=s_{i_{p-1}}'$. Thus, both the single descent statement and the exact values of $s_{i_p},s_{i_p}'$ go through.

As a result, we see that $x_1=A^{(k-1)}(s,s')$, $x_2=A^{(k-1)}(s',s)$, while $y_1=A^{(k)}(s,s')$, $y_2=A^{(k)}(s',s)$. So we are done.
\end{proof}

The following lemma is then straightforward.
\begin{lemma}\label{lem:right-angle-descent-down}
Let $x_1,x_2\lessdot y_1,y_2$ form a butterfly in a right-angled Coxeter group $W$. If $q\in D_L(y_1)\cap D_L(y_2)$, then $q\in D_L(x_1)\cap D_L(x_2)$.
\end{lemma}
\begin{proof}
Write $x_1,x_2,y_1,y_2$ in the form as in Lemma~\ref{lem:butterfly-right-angle-structure}. Pick any reduced word of $y_1$ compatible with the decomposition $y_1=u\cdot A^{(m+1)}(s,s')\cdot v$. If $q\in D_L(u)$, then we clearly have $q\in D_L(x_1)\cap D_L(x_2)$. Similarly, if the first appearance of $q$ is inside $v$, meaning that $q$ commutes with the simple generators before it in $v$, and with $s$ and $s'$, and with $u$, by Lemma~\ref{lem:descent-if-can-move}, then we have $q\in D_L(x_1)$ and $q\in D_L(x_2)$ as well. Lastly, if the first appearance of $q$ in $y_1$ is inside $A^{(m+1)}(s,s')$, meaning that $q=s$ and $s$ does not appear in $u$, then $q$ cannot be a left descent of $y_2$ by Lemma~\ref{lem:descent-if-can-move}, a contradiction.
\end{proof}

It is now clear that for a butterfly $x_1,x_2\lessdot y_1,y_2$ in a right-angled Coxeter group, $y_1$ and $y_2$ have (at least, and in fact) two upper covers which are $u\cdot A^{(m+2)}(s,s')\cdot v$ and $u\cdot A^{(m+2)}(s',s)\cdot v$ with notations as in Lemma~\ref{lem:butterfly-right-angle-structure}. We are now ready to prove the main theorem of this section, which resolves the right-angled case in Lemma~\ref{lem:butterfly}.

\begin{lemma}\label{lem:butterfly-right-angle-join}
Let $x_1,x_2\lessdot y_1,y_2$ form a butterfly in a right-angled Coxeter group $W$. If $w\geq y_1,y_2$, then there exists some $z$ which cover both $y_1$ and $y_2$ such that $w\geq z$.
\end{lemma}
\begin{proof}
Use induction on $\ell(x_1)$, and on top of that, use induction on $\ell(w)$. Write $x_1=u\cdot A^{(m)}(s,s')\cdot v$, $x_2=u\cdot A^{(m)}(s',s)\cdot v$, $y_1=u\cdot A^{(m+1)}(s,s')\cdot v$, $y_2=u\cdot A^{(m+1)}(s',s)\cdot v$ as in Lemma~\ref{lem:butterfly-right-angle-structure}.

Take any left descent $q\in D_L(w)$. If $q\in D_L(y_1)\cap D_L(y_2)$, by Lemma~\ref{lem:right-angle-descent-down}, $q\in D_L(x_1)\cap D_L(x_2)$. This means that we can consider the butterfly $qx_1,qx_2\lessdot qy_1,qy_2$, with $qw\geq qy_1,qy_2$, via either the Subword Property or the Lifting Property of the strong order. By the induction hypothesis, there exists $z'\gtrdot qy_1,qy_2$ such that $qw\geq z'$. Now, $z=qz'$ is what we want. Similarly, if $q\notin D_L(y_1),D_L(y_2)$, then $qw\geq y_1,y_2$ so by the induction hypothesis on $\ell(w)$, we have $w\geq qw\geq z\gtrdot y_1,y_2$ as desired.

For the critical case, assume $q\in D_L(y_1)$ and $q\notin D_L(y_2)$. Pick any reduced reduced word of $y_1$ compatible with the decomposition $y_1=u\cdot A^{(m+1)}(s,s')\cdot v$. If the first (leftmost) appearance of $q$ is in $u$, then $q\in D_L(y_2)$, a contradiction. If the first appearance of $q$ is in the part of $v$, then $q$ can be moved all the way past $s$ and $s'$ and $u$, meaning that $q\in D_L(y_2)$, a contradiction. Thus, the first appearance of $q$ in $y_1$ is in the part of $A^{(m+1)}(s,s')$. This means that $q=s$, and that $s$ commutes with $u$. Since $w\geq y_2$, $q\in D_L(w)$, $q\notin D_L(y_2)$, the Lifting Property says that $w\geq qy_2$. At the same time, \[qy_2=s\cdot u\cdot A^{(m+1)}(s',s)\cdot v=u\cdot A^{(m+2)}(s,s')\cdot v.\]
Let $z=qy_2$ and we see that $w\geq z\gtrdot y_1,y_2$ as desired.
\end{proof}

\section*{Acknowledgements}
We are very grateful to Thomas Lam and Grant Barkley for their helpful comments and suggestions. We also wish to thank Mario Marietti for alerting us to important references.

\bibliographystyle{plain}
\bibliography{main.bib}

\begin{thebibliography}{10}

\bibitem{SM-topology}
Nancy Abdallah, Mikael Hansson, and Axel Hultman.
\newblock Topology of posets with special partial matchings.
\newblock {\em Adv. Math.}, 348:255--276, 2019.

\bibitem{billey-postnikov}
Sara Billey and Alexander Postnikov.
\newblock Smoothness of {S}chubert varieties via patterns in root subsystems.
\newblock {\em Adv. in Appl. Math.}, 34(3):447--466, 2005.

\bibitem{bjorner-brenti}
Anders Bj\"{o}rner and Francesco Brenti.
\newblock {\em Combinatorics of {C}oxeter groups}, volume 231 of {\em Graduate
  Texts in Mathematics}.
\newblock Springer, New York, 2005.

\bibitem{blundell2021towards}
Charles Blundell, Lars Buesing, Alex Davies, Petar Veli{\v{c}}kovi{\'c}, and
  Geordie Williamson.
\newblock Towards combinatorial invariance for {K}azhdan-{L}usztig polynomials.
\newblock {\em arXiv preprint arXiv:2111.15161}, 2021.

\bibitem{Brenti-combinatorial-formula}
Francesco Brenti.
\newblock A combinatorial formula for {K}azhdan-{L}usztig polynomials.
\newblock {\em Invent. Math.}, 118(2):371--394, 1994.

\bibitem{SM-original}
Francesco Brenti.
\newblock The intersection cohomology of {S}chubert varieties is a
  combinatorial invariant.
\newblock {\em European J. Combin.}, 25(8):1151--1167, 2004.

\bibitem{SM-diamonds}
Francesco Brenti, Fabrizio Caselli, and Mario Marietti.
\newblock Diamonds and {H}ecke algebra representations.
\newblock {\em Int. Math. Res. Not.}, pages Art. ID 29407, 34, 2006.

\bibitem{SM-advances}
Francesco Brenti, Fabrizio Caselli, and Mario Marietti.
\newblock Special matchings and {K}azhdan-{L}usztig polynomials.
\newblock {\em Adv. Math.}, 202(2):555--601, 2006.

\bibitem{Carrell-smoothness}
James~B. Carrell.
\newblock The {B}ruhat graph of a {C}oxeter group, a conjecture of {D}eodhar,
  and rational smoothness of {S}chubert varieties.
\newblock In {\em Algebraic groups and their generalizations: classical methods
  ({U}niversity {P}ark, {PA}, 1991)}, volume~56 of {\em Proc. Sympos. Pure
  Math.}, pages 53--61. Amer. Math. Soc., Providence, RI, 1994.

\bibitem{SM-first-lower-classification}
Fabrizio Caselli and Mario Marietti.
\newblock Special matchings in {C}oxeter groups.
\newblock {\em European J. Combin.}, 61:151--166, 2017.

\bibitem{SM-lower-classification}
Fabrizio Caselli and Mario Marietti.
\newblock A simple characterization of special matchings in lower {B}ruhat
  intervals.
\newblock {\em Discrete Math.}, 341(3):851--862, 2018.

\bibitem{davies2021advancing}
Alex Davies, Petar Veli{\v{c}}kovi{\'c}, Lars Buesing, Sam Blackwell, Daniel
  Zheng, Nenad Toma{\v{s}}ev, Richard Tanburn, Peter Battaglia, Charles
  Blundell, Andr{\'a}s Juh{\'a}sz, et~al.
\newblock Advancing mathematics by guiding human intuition with ai.
\newblock {\em Nature}, 600(7887):70--74, 2021.

\bibitem{Dyer-hecke-algebras}
M.~J. Dyer.
\newblock Hecke algebras and shellings of {B}ruhat intervals.
\newblock {\em Compositio Math.}, 89(1):91--115, 1993.

\bibitem{dyer-bruhat-graph}
Matthew Dyer.
\newblock On the ``{B}ruhat graph'' of a {C}oxeter system.
\newblock {\em Compositio Math.}, 78(2):185--191, 1991.

\bibitem{self-dual}
Christian Gaetz and Yibo Gao.
\newblock Self-dual intervals in the {B}ruhat order.
\newblock {\em Selecta Math. (N.S.)}, 26(5):Paper No. 77, 23, 2020.

\bibitem{Gasharov}
Vesselin Gasharov.
\newblock Factoring the {P}oincar\'{e} polynomials for the {B}ruhat order on
  {$S_n$}.
\newblock {\em J. Combin. Theory Ser. A}, 83(1):159--164, 1998.

\bibitem{GKM}
Mark Goresky, Robert Kottwitz, and Robert MacPherson.
\newblock Equivariant cohomology, {K}oszul duality, and the localization
  theorem.
\newblock {\em Invent. Math.}, 131(1):25--83, 1998.

\bibitem{guillemin-holm-zara}
V.~Guillemin, T.~Holm, and C.~Zara.
\newblock A {GKM} description of the equivariant cohomology ring of a
  homogeneous space.
\newblock {\em J. Algebraic Combin.}, 23(1):21--41, 2006.

\bibitem{kazhdan-lusztig-polynomials}
David Kazhdan and George Lusztig.
\newblock Representations of {C}oxeter groups and {H}ecke algebras.
\newblock {\em Invent. Math.}, 53(2):165--184, 1979.

\bibitem{lakshmibai-sandhya}
V.~Lakshmibai and B.~Sandhya.
\newblock Criterion for smoothness of {S}chubert varieties in {${\rm
  Sl}(n)/B$}.
\newblock {\em Proc. Indian Acad. Sci. Math. Sci.}, 100(1):45--52, 1990.

\bibitem{SM-zircon}
Mario Marietti.
\newblock Algebraic and combinatorial properties of zircons.
\newblock {\em J. Algebraic Combin.}, 26(3):363--382, 2007.

\bibitem{Oh-Postnikov-Yoo}
Suho Oh, Alexander Postnikov, and Hwanchul Yoo.
\newblock Bruhat order, smooth {S}chubert varieties, and hyperplane
  arrangements.
\newblock {\em J. Combin. Theory Ser. A}, 115(7):1156--1166, 2008.

\bibitem{richmond-slofstra-fiber-bundle}
Edward Richmond and William Slofstra.
\newblock Billey-{P}ostnikov decompositions and the fibre bundle structure of
  {S}chubert varieties.
\newblock {\em Math. Ann.}, 366(1-2):31--55, 2016.

\bibitem{Tits-words}
Jacques Tits.
\newblock Le probl\`eme des mots dans les groupes de {C}oxeter.
\newblock In {\em Symposia {M}athematica ({INDAM}, {R}ome, 1967/68), {V}ol. 1},
  pages 175--185. Academic Press, London, 1969.

\bibitem{Waterhouse}
William~C. Waterhouse.
\newblock Automorphisms of the {B}ruhat order on {C}oxeter groups.
\newblock {\em Bull. London Math. Soc.}, 21(3):243--248, 1989.

\end{thebibliography}
\end{document}